\documentclass[final,onefignum,onetabnum]{siamart220329}

\usepackage{amsmath}
\usepackage{mathdots}
\usepackage{amsfonts}
\usepackage{amssymb}
\usepackage{latexsym}
\usepackage{footmisc}
\usepackage{graphicx}
\usepackage{pgfplots}
\usepackage{todonotes}
\usepackage{algorithm}
\usepackage{algpseudocode}
\usepackage{cleveref}
\usepackage[most]{tcolorbox}

\newcommand{\actionbox}[1]{\begin{tcolorbox}[colback=white,colframe=black,width=\columnwidth,boxsep=5pt,arc=4pt]
  #1
\end{tcolorbox}}
\usepackage{color}
\usepackage{verbatim}
\usepackage{url}
\usepackage{bm}

\def\mainalg{NoisyChebtrunc}

\newcommand{\diag}{{\rm diag}}

\DeclareMathOperator*{\minimize}{minimize}

\usepackage[normalem]{ulem}

\newcommand{\ignore}[1]{}
\newcommand{\argmin}{\mathop{\rm argmin}\limits}

\newcommand{\vct}[1]{\bm{#1}}
\newcommand{\mtx}[1]{\bm{#1}}
\newcommand{\T}{\bm{T}}

\title{
Polynomial approximation of noisy functions 
}
\author{Takeru Matsuda\thanks{Department of Mathematical Informatics, Graduate School of Information Science and Technology, The University of Tokyo, Japan. \& Statistical Mathematics Unit, RIKEN Center for Brain Science, Japan. (\email{matsuda@mist.i.u-tokyo.ac.jp}). Supported by JSPS KAKENHI Grant Numbers 22K17865, 24K02951 and JST Moonshot Grant Number JPMJMS2024.}
\and Yuji Nakatsukasa\thanks{Mathematical Institute, University of Oxford, Oxford, OX2 6GG, UK, (\email{nakatsukasa@maths.ox.ac.uk}). Supported by EPSRC grants EP/Y010086/1 and EP/Y030990/1.} }

\begin{document}
\maketitle

\begin{abstract}
	Approximating a univariate function on the interval $[-1,1]$ with a polynomial is among the most classical problems in numerical analysis. 
	When the function evaluations come with noise, a least-squares fit is known to reduce the effect of noise as more samples are taken. 
	The generic algorithm for the least-squares problem requires $O(Nn^2)$ operations, where $N+1$ is the number of sample points and $n$ is the degree of the polynomial approximant.
	This algorithm is unstable when $n$ is large, for example $n\gg \sqrt{N}$ for equispaced sample points. In this study, we blend numerical analysis and statistics to introduce a stable and fast $O(N\log N)$ algorithm called \mainalg\ based on Chebyshev interpolation.
	It has the same error reduction effect as least-squares and the convergence is spectral until the error reaches $O(\sigma \sqrt{{n}/{N}})$, where $\sigma$ is the noise level, after which the error continues to decrease at the Monte-Carlo $O(1/\sqrt{N})$ rate. 
	To determine the polynomial degree, \mainalg\ employs a statistical criterion, namely Mallows' $C_p$. 
	We analyze \mainalg\ in terms of the variance and concentration in the infinity norm to the underlying noiseless function. These results show that with high probability the infinity-norm error is bounded by a small constant times $\sigma \sqrt{{n}/{N}}$, when the noise {is} independent and follows a subgaussian or subexponential distribution.
	We illustrate the performance of \mainalg\ with numerical experiments.
\end{abstract}



\headers{Approximating noisy functions}{Takeru Matsuda and Yuji Nakatsukasa}

\section{Introduction}

Approximating a function $f$ on $[-1,1]$ by a polynomial is a classical and fundamental problem in numerical analysis\footnote{Throughout we will focus on the domain $[-1,1]$; this loses no generality as one can employ a trivial linear transformation to map any real interval $[a,b]$ to $[-1,1]$.}. 
Among the most successful  algorithms is Chebyshev interpolation~\cite[Ch.~2]{trefethenatap}, based on 
sampling $f$ at the Chebyshev points $x_i = \cos(i\pi/N)$ for $i=0,1,\ldots, N$, and finding a polynomial interpolant $\tilde p_N$ of degree $N$ such that $\tilde p_{N}(x_i)=f(x_i)$ for $i=0,1,\ldots, N$. 
Chebyshev interpolation combines \emph{speed} requiring only $O(N\log N)$ operations, \emph{stability} (the computation relies on the FFT, a unitary operation) and \emph{convergence} (essentially optimal Lebesgue constant, i.e., the error is within $O(\log N)$ of the best possible, resulting in spectral convergence, i.e., the smoother $f$ is, the faster~\cite[Ch.~7,8]{trefethenatap}; in particular, the convergence is exponential when $f$ is analytic on $[-1,1]$).

In classical approximation theory, it is assumed that the function $f$ can be evaluated exactly, i.e., without noise. 
In most cases in practice, however, evaluation of $f$ comes with noise, such as measurement or representation error. 
For example, the result of evaluation at $x$ may be given by 
\begin{equation}\label{eq:eval}
	y = f(x)   + \epsilon,\qquad \epsilon\sim \mathcal{N}(0,\sigma^2), 
\end{equation}
where $\mathcal{N}(0,\sigma^2)$ denotes the Gaussian distribution\footnote{Many of the results hold more generally for other noise distributions. We sometimes assume the noise is subgaussian or subexponential in our analysis; this will be made explicit.
	{We mostly assume that each evaluation of $f$ comes with \emph{independent} noise.}
} with mean 0 and (possibly unknown) variance $\sigma^2>0$.
Naturally, the ideal goal is to find ${p_*}$, the best approximation to $f$. 
Given the noise,  clearly no algorithm will be able to find $p_*$ exactly. But how can we approximate $f$ (or $p_*$) as accurately as possible? 
One would naturally hope to obtain an approximant with accuracy roughly equal to the noise level. It turns out that we can do much better. 

It is widely known in statistics and signal processing that performing a least-squares (LS) fitting/regression
in the noisy setting can help reduce the noise effect and avoid overfitting~\cite{dahlquist2008numerical,wasserman2013all}. Namely, to get an approximation $p_n(x)=\sum_{j=0}^n c_j \phi_j(x)$, where $\phi_0,\phi_1,\dots,\phi_n$ are the basis functions (not necessarily polynomials), one solves the least-squares problem 
\begin{equation}
	\label{eq:LSVander}
	\minimize_{\vct{c}} \|\mtx{V}\vct{c}-\vct{y}\|_2^2, 
\end{equation}
where $\mtx{V}\in\mathbb{R}^{(N+1)\times (n+1)}$ is the (generalized) Vandermonde matrix given by $V_{i,j}=\phi_{j-1}(x_{i-1})$. 
The solution of~\eqref{eq:LSVander} gives the coefficients $\vct{c}=[c_0,c_1,\ldots,c_n]^T$. 
While a number of papers have studied, analyzed and used least-squares methods for function approximation in the presence of noise \cite{belloni2015some,tsybakov,wasserman2006all},
to our knowledge, few studies focused on the classical and arguably most basic problem of approximating a univariate noisy function by a polynomial with deterministic sample points.
In addition, most studies do not optimize the computational complexity, presumably assuming a generic solver requiring $O(Nn^2)$ operations for the solution of the least-squares problem (or at least $O(Nn)$, if an iterative solver is used and the matrix is known to be well-conditioned).

In this paper, we propose {an $O(N\log N)$} method for polynomial approximation of a univariate noisy function, which we call \mainalg.
This method is based on truncating the Chebyshev interpolant at an appropriate degree and corresponds to solving a weighted least-squares problem. By taking advantage of the special structure that arises in Chebyshev interpolation, one can leverage many of the attractive properties (speed, stability and spectral convergence) to deal with the noisy case, while benefiting from statistical convergence results to reduce the noise effect (Monte-Carlo/central limit theorem (CLT) type noise reduction via sampling). 
We also employ a statistical tool, namely Mallows' $C_p$ \cite{mallows2000some}, to determine the polynomial degree in a data-driven fashion. Some discussion is given in~\cite{cohen2013stability} on such degree selection strategy when the noise level is unknown; however, details are not worked out there. 
While the idea of truncating  the Chebyshev interpolant is not at all a new idea and discussed e.g. in~\cite{aurentz2017chopping}, the effect of truncation, choice of degree and its analysis in the presence of noise has not been explored extensively. 

We also examine the convergence of \mainalg\ by blending classical numerical analysis tools (Chebyshev interpolation, Lebesgue constant etc) with concentration inequalities \cite{wainwright2019high}.
Specifically, we derive high-probability, non-asymptotic bounds with explicit constants for the convergence in the \emph{infinity} norm $\|\cdot \|_\infty$ to the unknown function $f$.
The convergence is at a spectral rate, until it reaches 
$O(\sigma \sqrt{{n}/{N}})$; note that the transition point depends on the number of sample points $N$ and noise level $\sigma$, 
and the error can be reduced further (but slowly) by increasing $N$, at the Monte Carlo rate of $O(1/\sqrt{N})$. 
Such results {in the $L_2$ norm} are implicit in e.g.~\cite{cohen2013stability,migliorati2015convergence} (where the focus is on choosing a good randomized sampling strategy), but not emphasized very much in the statistics literature, where the primary concern is the asymptotic behavior of convergence; see e.g. \cite{tsybakov}. 
We demonstrate through numerical experiments that the bounds we derive are reasonably indicative.

Overall, \mainalg\ 
combines (i) computational efficiency with $O(N\log N)$ operations, (ii) stability inherent in Chebyshev interpolation, leading to $L_\infty$ convergence results with high probability, 
and (iii) Monte-Carlo style noise reduction as more samples are taken. For large enough $N$, the error is $O(\sigma \sqrt{n/N})$. 
Note that Mallows' $C_p$ allows us to determine the polynomial degree without prior knowledge of the noise level $\sigma$.

We view this paper as a marriage of numerical analysis and statistics. Much of the paper uses standard tools from one of these subjects, but by putting things together we arrive at a powerful algorithm for polynomial approximation of noisy functions. 
Let us highlight the key innovations from each viewpoint:
\begin{itemize}
	\item In numerical analysis, the use of Chebyshev interpolation via the DCT is standard \cite{trefethenatap}. However, noise reduction 
	and approximation beyond the noise level is not often discussed, and the use of Mallows' $C_p$ for degree selection is not a standard tool in the field. While our algorithm is based on the basic tool of Chebyshev interpolation, to our knowledge, this paper is the first to show that its truncated version has attractive properties in the noisy setting. Recent papers \cite{cohen2013stability,cohen2,migliorati2015convergence} explore the convergence of approximation obtained by LS methods. These focus on the case where the sample points are random (drawn from a prescribed distribution), whereas in this paper we choose them deterministically to be the Chebyshev points, and highlight their attractive properties. In addition, many of these previous papers derive error bounds in the $L_2$ norm, while here we establish $L_{\infty}$ error bounds. 
	\item In statistics, 
	the problem of approximating an unknown function from noisy observations is classically discussed under the name of nonparametric regression~\cite[Ch.~5]{wasserman2006all}. 
	Many methods have been developed for this problem, such as kernel regression (Nadaraya--Watson estimate), local polynomials and splines \cite{tsybakov,wasserman2006all}, where the sample points are often equispaced. 
	Compared to approximants obtained by these methods, polynomial approximants by \mainalg\ are simpler to work with (e.g. to differentiate, integrate or find roots). 
	Also, \mainalg\ attains $O(N\log N)$ computational efficiency  by utilizing Chebyshev points as sample points.
	Note that \mainalg\ can be interpreted as the projection estimator with the Fourier basis \cite[Sec.~1.7]{tsybakov} applied to the periodic function $g(z)=f(\cos z)$ for $z \in [-\pi,\pi]$, where $f$ itself need not be periodic on $[-1,1]$\footnote{Here periodic means smoothness including the endpoints; functions like the Runge function $f(x)=1/(25x^2+1)$, which are periodic in terms of $f$, should not be regarded as periodic here as $f'$ is not periodic.}.

	Another classical technique is polynomial least-squares regression~\eqref{eq:LSVander} from equispaced samples (or those uniformly at random); 
	here the degree is usually low (e.g. bounded by the dozens~\cite[Sec.~1.1]{bishop2006pattern}). 
	This method can converge spectrally if the degree $n$ is chosen appropriately, and we compare it with \mainalg\ in the forthcoming discussions. 
	In numerical analysis, it is not unusual to take the degree in the thousands or even millions, as such degree may be necessary to achieve high accuracy for functions that are not smooth~\cite{trefethenatap}; and algorithms are available to make such computations feasible. 
	We will see that the usual polynomial least-squares can lead to stability issues when the degree is large (close to the number of sample points); an issue \mainalg\  overcomes. 
	Also note that \mainalg\ can be viewed as solving weighted least-squares problems, as shown in Section~\ref{sec:weight}.
\end{itemize}

{\it Notation}. Throughout, the observations are $N+1$ samples 
$\{(x_i,y_i)\}_{i=0}^N$, where as in~\eqref{eq:eval}, each evaluation is 
\begin{equation} \label{eq:yieval}  
	y_i=f(x_i)   + \epsilon_i, 
\end{equation}
and the noises $\{ \epsilon_i \}_{i=0}^N$ are independent random variables such as $\epsilon_i \sim \mathcal{N}(0,\sigma^2)$.
$n$ is the degree of the polynomial approximant. $T_i(x)$ denotes the $i$th Chebyshev polynomial of the first kind, and $\|\cdot\|_\infty$ denotes the infinity norm of functions on $[-1,1]$, so $\|f\|_\infty=\sup_{x\in[-1,1]}|f(x)|$. We use boldface lowercase letters to denote vectors, e.g.,~$\vct{y}=[y_0,y_1,\ldots,y_N]^T$ (vectors are in $\mathbb{R}^{N+1}$ with the exception of the coefficient vector $\vct{c} \in \mathbb{R}^{n+1}$), and boldface uppercase letters for matrices. $\mathbb{E}$ denotes the expected value over the random variables $\{ \epsilon_i \}_{i=0}^N$. 

\subsection{Motivation and illustration} 

{Let us motivate the algorithm by demonstrating that it is possible to obtain accuracy much higher than noise level $\sigma$.} Consider approximating the Runge function $f(x)=1/(25x^2+1)$, whose evaluations are contaminated by independent Gaussian noise $\epsilon_i \sim \mathcal{N}(0,\sigma^2)$ as in~\eqref{eq:yieval} with noise level $\sigma=10^{-4}$.
We compute the polynomial interpolants of $\{ y_i \}_{i=0}^N$ at $(N+1)$ Chebyshev points $\{ x_i \}_{i=0}^N$ by the DCT or FFT~\cite{strang1999discrete,trefethen2000spectral}, where we vary the degree $N\in \{2^5,2^{7},2^{22}\}$. 
{Owing to the $O(N\log N)$ complexity, each computation takes only a fraction of a second on a standard laptop.}

Figure~\ref{fig:coeffs0} (left) plots the magnitudes of the Chebyshev coefficients (available via Chebfun's {\tt plotcoeffs} command), 
that is, $|c_j|$ where $\tilde{p}_N(x) = \sum_{j=0}^N c_jT_{j}(x)$ is the polynomial interpolant of degree $N$. 
We do the same for the noiseless target function $f$ (approximated to $10^{-15}$ accuracy by a Chebyshev expansion), whose coefficients decay exponentially and forms a straight dotted line in the figure. Note that as $f$ is an even function, its odd-degree coefficients are all 0. 
The right panel of Figure~\ref{fig:coeffs0} plots the error $|f(x)-\tilde p_N(x)|$. 

\begin{figure}[htbp]
	\centering
	\begin{minipage}[t]{.495\linewidth}
		\includegraphics[width=.98\textwidth]{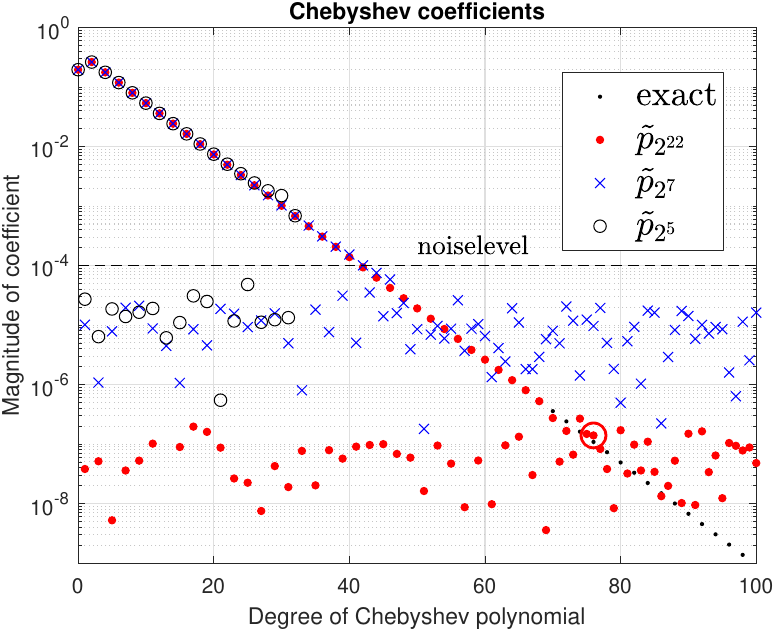}            
	\end{minipage}
	\begin{minipage}[t]{.495\linewidth}
		\includegraphics[width=.935\textwidth]{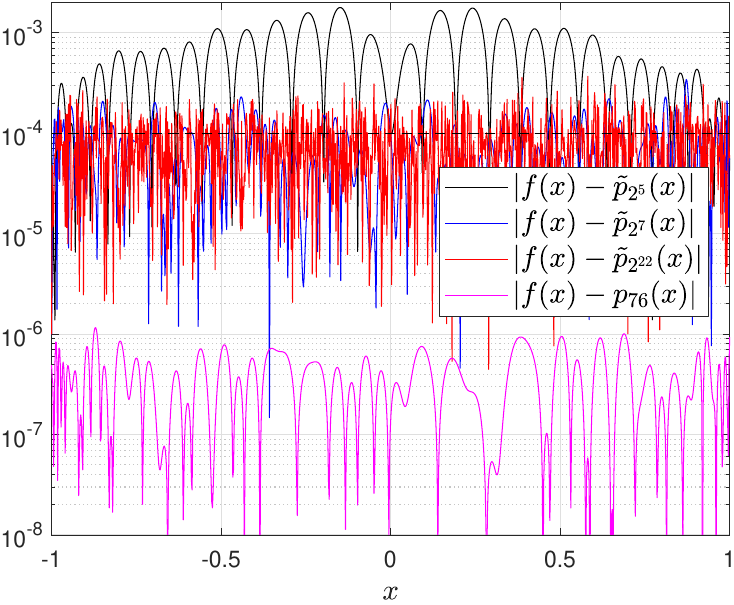}            
	\end{minipage}
	\caption{Left: Chebyshev coefficients (leading $100$ terms) of the noiseless function $f$ (black dots) and the Chebyshev interpolants $\tilde p_{2^d}$ of the noisy evaluations of $f$ of degrees $2^5,2^7$ and $2^{22}$. The red circle indicates the degree $76$ at which Mallows' $C_p$ truncates. 
		Right: Error plots $|f(x)-\tilde p_N(x)|$ for the interpolants with $N=2^5,2^7,2^{22}$, along with \mainalg's output $p_{76}$. 
	}
	\label{fig:coeffs0}
\end{figure}

Figure~\ref{fig:coeffs0} illustrates a number of phenomena worth highlighting. First, with the degree-$2^5$ interpolant $\tilde p_{2^5}$ (we use tilde for interpolants), the Chebyshev coefficients do not decay to the noise level $\sigma=10^{-4}$, and the error $|f(x)-\tilde p_{2^5}(x)|$ is dominated by the truncation/aliasing error in Chebyshev interpolation; i.e., the degree is not high enough. This is essentially the regime of classical approximation theory, and convergence is improved by increasing the degree; here one should take $n=N$. A sensible interpretation is that the polynomial degree should be (at least) $40$, where the Chebyshev coefficients reach the noise level. 

With the interpolants of degrees $2^7$ and $2^{22}$, the degree is taken large enough for $c_j$ to decay below noise level. However, we see a clear difference between the two: whereas the accuracy of the $\tilde p_{2^7}$ coefficients are good only up to around $c_{40}$, those of $\tilde p_{2^{22}}$ are accurate much further, up to around\footnote{A least-squares approach with $n=76$ gives very similar coefficients to those of $p_{2^{22}}$ up until $c_{76}$, with obviously no coefficients beyond $c_{77}$. We will explore such connections and differences further in what follows.}
$c_{70}$. 
The difference manifests itself also in terms of the magnitudes of the tail coefficients ($c_j$ for $j\geq 80$; equivalently, as the odd coefficients are 0 for the exact function $f$, we can examine the magnitudes of $c_{2k+1}$ for any $k$): those of $p_{2^7}$ are roughly on the order of $10^{-5}$, while those of $\tilde p_{2^{22}}$ are about $10^{-7}$. 
This behavior, that higher-degree interpolants have better accuracy in individual coefficients, is a general phenomenon, and can be observed regardless of the particular observations of the noise, or the choice of $f$. 

Given these, one might expect $\tilde p_{2^{22}}$ to better approximate $f$ than $\tilde p_{2^7}$. However, this is not the case. From the right panel of Figure~\ref{fig:coeffs0}, we see that the comparison between 
$|f(x)-\tilde p_{2^7}(x)|$ and $|f(x)-\tilde p_{2^{22}}(x)|$ is inconclusive, 
and that both errors are roughly on the noise level $10^{-4}$ rather than the improved accuracy suggested by the plot of coefficients in the left panel of Figure~\ref{fig:coeffs0}. 
This might seem contradictory, but the explanation is that the tail coefficients $c_j$ for $j>80$ collectively contribute to error on the noise level: there are more than $4$ million such $c_j$ for $\tilde p_{2^{22}}$, while only $58$ for $\tilde p_{2^7}$; these are mostly a result of the noise, and have little to do with $f$. 
In Section~\ref{sec:var} we quantify how taking $n$ too large can impair the accuracy. 
Arguably, here $\tilde p_{2^{22}}$ is a worse approximant than $\tilde p_{2^{7}}$ as it has a higher degree, and hence more costly to work with (e.g. to differentiate, integrate or find roots, for the downstream application). 

Fortunately, there is a natural and simple workaround: truncate the Chebyshev expansion of $\tilde p_{2^{22}}$. 
A systematic way to choose an appropriate truncation degree using Mallows' $C_p$ criterion will be discussed in Section~\ref{sec:degselect}; for the moment, we truncate it to degree $76$ (which $C_p$ chooses) and call it $p_{76}$; 
the lack of tilde indicates that it is not an interpolant at the sample points $\{x_i\}_{i=0}^N$. 
From the right panel of Figure~\ref{fig:coeffs0}, the polynomial $p_{76}$ visibly has much better accuracy than the rest\footnote{One can further improve the accuracy slightly by zeroing out the odd-degree coefficients of $p_{76}$, which are clearly artifacts of the noise because the function $f$ is even. This is related to what is known as sparse estimation in statistics \cite{hastie2015statistical}.
	\label{fnlabel}}, and has error $|f(x)-p_{76}(x)| \approx 10^{-6}$, achieving error reduction of two orders of magnitude relative to $\tilde p_{2^{22}}$, {i.e., the noise level $\sigma$}. 
Overall, the construction of $p_{76}$ is an instance of our proposed algorithm \mainalg: 
Sample $f$ (with noise) at as many Chebyshev points as possible. Then find the polynomial interpolant, and truncate its Chebyshev expansion at degree $n$ chosen by Mallows' $C_p$. 
The main purpose of this paper is to explain why this process yields a good approximation of $f$. 

\section{Algorithm} The essence of our algorithm \mainalg\ is laid out in the preceding example: Chebyshev interpolation followed by truncation. The remaining question is how to choose the degree. 
For this purpose, we employ Mallows' $C_p$ \cite{mallows2000some} in statistics, which was originally developed for least-squares estimates like~\eqref{eq:LSVander}.
We derive Mallows' $C_p$ for \emph{weighted} least-squares estimates because, as shown below, Chebyshev interpolation can be viewed as the solution of a (series of) weighted least-squares problem. 
We therefore first show that \mainalg\ is mathematically (but not computationally; \mainalg\ is more efficient) equivalent to weighted least-squares. 

\subsection{Weighted least-squares and \mainalg}\label{sec:weight}
In what follows we always take the basis functions in \eqref{eq:LSVander} to be Chebyshev polynomials $\phi_j(x)=T_j(x)$ on $[-1,1]$ and $\vct{x}=[x_0,\ldots,x_N]^T$ are Chebyshev points $x_i = \cos(i\pi/N)$.
The least-squares problem is expressed as  
\begin{equation}  \label{eq:Tc-f}
	\minimize_{\vct{c}} \|\mtx{T}\vct{c}-\vct{y}\|_2^2,
\end{equation}
where $\mtx{T}$ is an $(N+1)\times (n+1)$ matrix with elements $\mtx{T}_{i,j} = T_{j-1}(x_{i-1})$, that is, $\T$ 
is the Vandermonde matrix with respect to the Chebyshev polynomials of the first kind, and the degree $n<N$ is usually prescribed, or at least an upper bound on $n$ is given. 
Having solved \eqref{eq:Tc-f} to obtain 
$\vct{c} = (\mtx{T}^T\mtx{T})^{-1}\mtx{T}^T\vct{y}$, 
we obtain a polynomial approximant $\hat{p}_n(x)=\sum_{j=0}^n c_j T_j(x)$.

As a minor but important generalization, one can also consider a \emph{weighted} least-squares problem 
\begin{equation}  \label{eq:DTc-f}
	\minimize_{\vct{c}} \|\mtx{D}(\T\vct{c}-\vct{y})\|_2^2
\end{equation} 
for a diagonal matrix $\mtx{D} \in \mathbb{R}^{(N+1) \times (N+1)}$. 
Its solution is 
\[
\vct{c}=((\mtx{D} \mtx{T})^T (\mtx{D} \mtx{T}))^{-1} (\mtx{D} \mtx{T})^T \mtx{D} \vct{y}=(\mtx{T}^T \mtx{D}^2 \mtx{T})^{-1} \mtx{T}^T \mtx{D}^2 \vct{y}.
\]
We note a simple but important fact: 

\begin{lemma}\label{lem:equiv}
	Let 
	$\vct{x}=[x_0,\ldots,x_N]^T$ be Chebyshev points $x_i = \cos(i\pi/N)$ and $\vct{y}=[y_0,\ldots,y_N]^T$, where $y_i\in\mathbb{R}$. Define
	\begin{equation}  \label{eq:defD}
		\mtx{D}=\mbox{diag}\left( \frac{1}{\sqrt{2}},1,1,\ldots,1,\frac{1}{\sqrt{2}} \right) \in \mathbb{R}^{(N+1) \times (N+1)}. 
	\end{equation}
	Then for each $n=0,\ldots, N$, the solution of \eqref{eq:DTc-f} is equal to $\vct{c}=[c_0,\ldots,c_n]
	$, where 
	$\tilde p_{N}(x)=\sum_{j=0}^N c_jT_{j}(x)$ is the unique interpolant of $\vct{y}$ at $\vct{x}$. 
\end{lemma}

\begin{proof}  
	We first note a discrete orthogonality property of Chebyshev polynomials. 
	Let $\mtx{T}_N\in\mathbb{R}^{(N+1)\times (N+1)}$ be the \emph{square}  Chebyshev Vandermonde matrix, or equivalently, the matrix $\mtx{T}$ when $n=N$ (that is, $\mtx{T}$ is the first $n+1$ columns of $\mtx{T}_N$). Then these matrices satisfy~\cite[\S~4.6]{mason2010chebyshev} 
	$(\mtx{DT}_N)^T(\mtx{DT}_N)=\frac{N}{2}\mtx{D}^{-2}$, so $\mtx{DT}_N$ has orthogonal columns, and hence so does $\mtx{DT}$. 
	
	Now consider the linear system for interpolation
	\begin{equation}
		\label{eq:interp}
		\mtx{T}_N\vct{c}_N=\vct{y} \quad \Leftrightarrow \quad 
		\mtx{DT}_N\vct{c}_N=\vct{Dy}. 
	\end{equation}
	We focus on the latter weighted version, as it is the matrix $\mtx{DT}_N$ that has orthogonal columns. The solution is 
	\begin{equation}  \label{eq:cN}
		\vct{c}_N=
		((\mtx{DT}_N)^T\mtx{DT}_N)^{-1}(\mtx{DT}_N)^T\mtx{Dy}
		=\frac{2}{N}\mtx{D}^2\mtx{T}_N^T\!\mtx{D}^2\vct{y}.   
	\end{equation}
	\[
	\]
	It remains to show that these are exactly the coefficients one obtains by solving the LS problem~\eqref{eq:DTc-f}; more precisely, the $i$th element of $\vct{c}_N$ and $\vct{c}$ are equal. 
	To do so, we use the basic fact that for any LS problem of the form $\minimize_{\vct{x}_1,\vct{x}_2}\left\|[\mtx{M}_1\ \mtx{M}_2]
	\begin{bmatrix}
		\vct{x}_1\\\vct{x}_2\end{bmatrix}-\vct{b}\right\|_2$
	where $\mtx{M}_1$ and $\mtx{M}_2$ are orthogonal 
	$\mtx{M}_1^T\mtx{M}_2=0$, the problem can be decoupled and we have $\vct{x}_1 = \argmin\left\|\mtx{M}_1\vct{x}_1-\vct{b}\right\|_2$ and 
	$\vct{x}_2 = \argmin\left\|\mtx{M}_2\vct{x}_2-\vct{b}\right\|_2$. Finally note that $\mtx{DT}$ is simply the leading $n+1$ columns of $\mtx{DT}_N$. 
\end{proof}

Lemma~\ref{lem:equiv} implies that we can solve \eqref{eq:DTc-f} for all $n$ by performing Chebyshev interpolation, which, in turn, can be done in $O(N\log N)$ operations using the FFT (or DCT\footnote{One can alternatively use the DCT-2 transformation, which has the benefit of resulting in a Chebyshev Vandermonde matrix that has orthogonal columns. Here we use DCT-1 as it allows us to use Chebfun's functions effortlessly.}), as is well known in numerical analysis~\cite{strang1999discrete,trefethen2000spectral}. 
It also benefits our development of \mainalg: we are now able to introduce Mallows' $C_p$ for weighted least-squares problems for degree selection, as discussed in the next subsection.

\subsection{Degree selection by Mallows' $C_p$}\label{sec:degselect}

We employ an extension of Mallows' $C_p$ \cite{mallows2000some,james2013introduction} for selecting the polynomial degree $n$ based on the observation $(x_0,y_0),\dots,(x_N,y_N)$. 
Mallows' $C_p$ is a classical criterion in statistics for evaluating the goodness of fit of a regression model estimated by least squares, and hence can be used to select variables (in our context, the degree $n$).
When the error is Gaussian and the noise level $\sigma^2$ is known, it is equivalent to the famous Akaike Information Criterion (AIC) \cite{akaike1998information,konishi2008information}. 
Notably, Mallows' $C_p$ does not require the noise level $\sigma^2$ to be known.
Here, we describe
Mallows' $C_p${, slightly generalized} to the context of weighted least-squares for degree selection in \mainalg.

We summarize the notation to be used in Lemma~\ref{lem_Cp} below.
Let $\bar{n}$ be some upper bound of the polynomial degree\footnote{We set $\bar{n}=\lfloor (N+1)/2\rfloor$ in our experiments. 
}. For $\ell=0,1,2,\dots,\bar{n}$, let
\begin{align*}
	\hat{\vct{c}}_{\ell} &= \argmin_{\vct{c}_{\ell}} \| \mtx{D}(\mtx{T}_{\ell} \vct{c}_{\ell} -\vct{y}) \|_2^2 \\
	&= (\mtx{T}_{\ell}^T \mtx{D}^2 \mtx{T}_{\ell})^{-1} \mtx{T}_{\ell}^T \mtx{D}^2 \vct{y} \in \mathbb{R}^{\ell+1}
\end{align*}
be the weighted least-squares estimate of the Chebyshev coefficients for degree $\ell$ obtained by {\mainalg}, where $\mtx{T}_{\ell}$ is the $(N+1) \times ({\ell}+1)$ matrix that consists of the first $(\ell+1)$ columns of $\mtx{T}_N$.
From Lemma~\ref{lem:equiv}, $\hat{\vct{c}}_{\ell}$ is equal to the first $(\ell+1)$ coefficients of Chebyshev interpolation.
Also, let
\[
\hat{\sigma}^2 = \frac{1}{N-\bar{n}} \| 
\mtx{{D}}(
\mtx{T}_{\bar{n}} \hat{\vct{c}}_{\bar{n}} - \vct{y}) \|_2^2
\]
be an estimate of the error variance $\sigma^2$, which is approximately unbiased when the error is Gaussian and the function $f$ is a polynomial of degree $\bar{n}$ (well-specified setting). 
By using Lemma~\ref{lem:equiv} and $(\mtx{DT}_N)^T(\mtx{DT}_N)=\frac{N}{2}\mtx{D}^{-2}=\frac{N}{2}\mbox{diag}( 2,1,1,\ldots,1,2 )$, we have
\[
\hat{\sigma}^2 = \frac{N}{2(N-\bar{n})} (\| \hat{\vct{c}}_{\bar{n}+1:N} \|_2^2 + \hat{c}_{N}^2),
\]
where $\hat{\vct{c}}_{\bar{n}+1:N}$ denotes the last $(N-\bar{n})$ entries of $\vct{c}$ so that $\vct{c} = [\hat{\vct{c}}_{\bar{n}}^T \ \hat{\vct{c}}_{\bar{n}+1:N}^T]^T$.

The degree selection problem can be viewed as a special case of variable selection in linear regression, for which Mallows' $C_p$ is widely used \cite{james2013introduction,mallows2000some}.
However, the classical form of Mallows' $C_p$ is for unweighted LS estimates, not weighted LS problems as used by \mainalg\ (Lemma~\ref{lem:equiv}).
Thus, for completeness, we describe a natural extension of Mallows' $C_p$ to general linear estimates and general quadratic loss in the following lemma {(we will take $\mtx{M}=\mtx{D}^2$ and $\mtx{B}=\mtx{T}_{\ell} (\mtx{T}_{\ell}^T \mtx{D}^2 \mtx{T}_{\ell})^{-1} \mtx{T}_{\ell}^T \mtx{D}^2$)}. 

\begin{lemma}\label{lem_Cp}
	Let $\vct{y}$ be an $n$-dimensional random vector with mean $\vct{\mu}$ and covariance $\sigma^2 \mtx{I}_n$.
	Let $\hat{\vct{\mu}}=\mtx{B} \vct{y}$ be a linear estimate of $\vct{\mu}$ and $\hat{\sigma}^2$ be an unbiased estimate of $\sigma^2$.
	For an independent copy $\widetilde{\vct{y}}$ of $\vct{y}$, 
	\begin{align*}
		C_p = \| \hat{\vct{\mu}}-\vct{y} \|_{\mtx{M}}^2 + \hat{\sigma}^2 {\rm tr} (\mtx{M} (\mtx{B}+\mtx{B}^{\top}))
	\end{align*}
	is an unbiased estimate of ${\rm E} [ \| \widetilde{\vct{y}}-\hat{\vct{\mu}} \|_{\mtx{M}}^2]$, where $\mtx{M}$ is a $n \times n$ positive definite matrix and $\| \vct{z} \|_{\mtx{M}} = (\vct{z}^T \mtx{M} \vct{z})^{1/2}$.
\end{lemma}

\begin{proof}
	From
	\begin{align*}
		{\mathbb{E}} [ \| \hat{\vct{\mu}}-\vct{y} \|_{\mtx{M}}^2] &= {\mathbb{E}} [\vct{y}^{\top} (\mtx{I}_n-\mtx{B})^{\top} \mtx{M} (\mtx{I}_n-\mtx{B}) \vct{y}] \\
		&= {\rm tr} ((\mtx{I}_n-\mtx{B})^{\top} \mtx{M} (\mtx{I}_n-\mtx{B}) {\mathbb{E}} [\vct{y} \vct{y}^{\top}]) \\
		&= {\rm tr} ((\mtx{I}_n-\mtx{B})^{\top} \mtx{M} (\mtx{I}_n-\mtx{B}) (\vct{\mu} \vct{\mu}^{\top}+\sigma^2 \mtx{I}_n)) \\
		&= \| \vct{\mu} -\mtx{B} \vct{\mu} \|_{\mtx{M}}^2 + \sigma^2 {\rm tr} ((\mtx{I}_n-\mtx{B})^{\top} \mtx{M} (\mtx{I}_n-\mtx{B})),
	\end{align*}
	we obtain
	\begin{align*}
		{\mathbb{E}} [ \| {\vct{\mu}}-\hat{\vct{\mu}} \|_{\mtx{M}}^2 ] &= {\mathbb{E}} [ \| {\vct{\mu}}-\mtx{B} \vct{\mu}+\mtx{B} \vct{\mu}-\mtx{B} \vct{y} \|_{\mtx{M}}^2 ] \\
		&= {\mathbb{E}} [\| {\vct{\mu}}-\mtx{B} \vct{\mu} \|_{\mtx{M}}^2] + {\mathbb{E}} [\| \mtx{B} \vct{\mu}-\mtx{B} \vct{y} \|_{\mtx{M}}^2] \\
		&= \| {\vct{\mu}}-\mtx{B} \vct{\mu} \|_{\mtx{M}}^2 + \sigma^2 {\rm tr} (\mtx{B} \mtx{B}^{\top} \mtx{M}) \\
		&= {\mathbb{E}} [\| \hat{\vct{\mu}}-\vct{y} \|_{\mtx{M}}^2] - \sigma^2 {\rm tr} ((\mtx{I}_n-\mtx{B})^{\top} \mtx{M} (\mtx{I}_n-\mtx{B})) + \sigma^2 {\rm tr} (\mtx{B} \mtx{B}^{\top} \mtx{M}),
	\end{align*}
	where we used ${\mathbb{E}} [\| \vct{z} \|_{\mtx{M}}^2] = {\rm tr} (\mtx{\Sigma} \mtx{M})$ for a random vector $\vct{z}$ with mean zero and covariance $\mtx{\Sigma}$.
	Therefore,
	\begin{align*}
		{\mathbb{E}} [\| \widetilde{\vct{y}}-\hat{\vct{\mu}} \|_{\mtx{M}}^2] &= {\mathbb{E}} [\| \widetilde{\vct{y}}-\vct{\mu}+\vct{\mu}-\hat{\vct{\mu}} \|_{\mtx{M}}^2] \\
		&={\mathbb{E}} [\| \widetilde{\vct{y}}-{\vct{\mu}} \|_{\mtx{M}}^2] + {\mathbb{E}} [\| {\vct{\mu}}-\hat{\vct{\mu}} \|_{\mtx{M}}^2] \\
		&= \sigma^2 {\rm tr} (\mtx{M}) +{\mathbb{E}} [\| \hat{\vct{\mu}}-\vct{y} \|_{\mtx{M}}^2] - \sigma^2 {\rm tr} ((\mtx{I}_n-\mtx{B})^{\top} \mtx{M} (\mtx{I}_n-\mtx{B})) + \sigma^2 {\rm tr} (\mtx{B} \mtx{B}^{\top} \mtx{M}) \\
		&= \sigma^2 {\rm tr} (\mtx{M} (\mtx{B}+\mtx{B}^{\top})) + {\mathbb{E}} [\| \hat{\vct{\mu}}-\vct{y} \|_{\mtx{M}}^2] \\
		&= {\mathbb{E}} [C_p],
	\end{align*}
	which shows that $C_p$ is an unbiased estimate of $\| \widetilde{\vct{y}}-\hat{\vct{\mu}} \|_{\mtx{M}}^2$. 
\end{proof}

When $\mtx{B}$ is an orthogonal projection matrix and $\mtx{D}=\mtx{I}_n$, Lemma~\ref{lem_Cp} reduces to the usual theory of Mallows' $C_p$ for least-squares estimates \cite{james2013introduction}.
Note that \cite{yanagihara2010unbiased} extended Mallows' $C_p$ to ridge regression, which corresponds to a specific choice of $\mtx{B}$ in Lemma~\ref{lem_Cp}.

Then, by setting $\mtx{B}=\mtx{T}_{\ell} (\mtx{T}_{\ell}^T \mtx{D}^2 \mtx{T}_{\ell})^{-1} \mtx{T}_{\ell}^T \mtx{D}^2$ and $\mtx{M}=\mtx{D}^2$ in Lemma~\ref{lem_Cp}, we define Mallows' $C_p$ for the polynomial degree $\ell$ by
\begin{align}\label{Cpdef}
	C_p(\ell) = \|\mtx{D} (\mtx{T}_{\ell} \hat{\vct{c}}_{\ell} - \vct{y})\|_2^2 + 2  \hat{\sigma}^2 {\rm tr} (\mtx{D}^2 \mtx{T}_{\ell} (\mtx{T}_{\ell}^T \mtx{D}^2 \mtx{T}_{\ell})^{-1} \mtx{T}_{\ell}^T \mtx{D}^2).
\end{align}
Qualitatively, the first term corresponds to the goodness of fit and becomes smaller for larger $\ell$, while the second term penalizes large degree (which induces over-fitting) and  increases with $\ell$.
Thus, minimization of $C_p$ attains a good trade-off between model fit and model complexity \cite{james2013introduction}.

Let us simplify the expression \eqref{Cpdef} to enable efficient computation. 
First, consider the first term of \eqref{Cpdef}.
Since 
$\mtx{D} \mtx{T}_{N}$ has orthogonal columns and $\mtx{D} \mtx{T}_{N} \vct{c}_N = \mtx{D} \vct{y}$ from \eqref{eq:interp}, we have
$\mtx{D} (\mtx{T}_{\ell} \hat{\vct{c}}_{\ell} - \vct{y})
=-\mtx{D} \mtx{T}_{\ell+1:N} \hat{\vct{c}}_{\ell+1:N}$, 
where $\mtx{T}_{\ell+1:N}$ denotes the last $(N-\ell)$ columns of $\mtx{T}_N$ so that $\mtx{T}_N = [\mtx{T}_{\ell} \ \mtx{T}_{\ell+1:N}]$ and $\hat{\vct{c}}_{\ell+1:N}$ denotes the last $(N-\ell)$ entries of $\vct{c}_N$ so that 
$\vct{c}_N = [\hat{\vct{c}}_{\ell}^T \ \hat{\vct{c}}_{\ell+1:N}^T]^T$.
Then, since 
$\mtx{D} \mtx{T}_{\ell+1:N}$ has orthogonal columns with norms $\sqrt{{N}/{2}}$ (except the final column with norm {$\sqrt{N}$}), 
\begin{align*}
	\|\mtx{D} (\mtx{T}_{\ell} \hat{\vct{c}}_{\ell} - \vct{y})\|_2^2 = \| \mtx{D} \mtx{T}_{\ell+1:N} \hat{\vct{c}}_{\ell+1:N} \|_2^2 = \frac{N}{2}(\|\hat{\vct{c}}_{\ell+1:N}\|_2^2+\hat{c}_{N}^2). 
\end{align*}
Next, consider the second term of \eqref{Cpdef}.
The trace can be rewritten as
\begin{align*}
	{\rm tr} ((\mtx{T}_{\ell}^T \mtx{D}^2 \mtx{T}_{\ell})^{-1} \mtx{T}_{\ell}^T \mtx{D}^4\mtx{T}_{\ell} )  &=
	{\rm tr} ((\mtx{T}_{\ell}^T \mtx{D}^2 \mtx{T}_{\ell})^{-1}
	(\mtx{T}_{\ell}^T (\mtx{D}^2 + (\mtx{D}^4 -\mtx{D}^2 ) ) \mtx{T}_{\ell} )  \\
	&=
	{\rm tr} (\mtx{I}_{\ell+1}-
	(\mtx{T}_{\ell}^T \mtx{D}^2 \mtx{T}_{\ell})^{-1}
	\mtx{T}_{\ell}^T (\mtx{D}^2-\mtx{D}^4 )  \mtx{T}_{\ell} ) \\
	&=
	\ell+1 - {\rm tr} (
	(\mtx{T}_{\ell}^T \mtx{D}^2 \mtx{T}_{\ell})^{-1}
	\mtx{T}_{\ell}^T (\mtx{D}^2-\mtx{D}^4 )  \mtx{T}_{\ell} ).
\end{align*}
Since $(\mtx{DT}_N)^T(\mtx{DT}_N)=\frac{N}{2}\mtx{D}^{-2}$ and $\mtx{T}_{\ell}^T \mtx{D}^2 \mtx{T}_{\ell}$ is its leading 
$(\ell+1)\times (\ell+1)$ part, we have 
$\mtx{T}_{\ell}^T \mtx{D}^2 \mtx{T}_{\ell}=\frac{N}{2}\mbox{diag}(2,1,\ldots,1,1)$ and $(\mtx{T}_{\ell}^T \mtx{D}^2 \mtx{T}_{\ell})^{-1}=
\frac{2}{N}\mbox{diag}(\frac{1}{2},1,\ldots,1,1)$ when $\ell<N$. 
Also, since the first and last row of $\mtx{T}_{\ell}$ are all $\pm 1$s, the diagonal elements of
$
\mtx{T}_{\ell}^T (\mtx{D}^2 -\mtx{D}^4 )  \mtx{T}_{\ell} 
=
\mtx{T}_{\ell}^T \mbox{diag}(\frac{1}{4},0,\ldots,0,\frac{1}{4})  \mtx{T}_{\ell}$ are all $\frac{1}{2}$. Therefore,
\[
{\rm tr} ((\mtx{T}_{\ell}^T \mtx{D}^2 \mtx{T}_{\ell})^{-1}
\mtx{T}_{\ell}^T (\mtx{D}^2-\mtx{D}^4 )  \mtx{T}_{\ell} )
=\frac{2}{N}\left(\frac{1}{4}+\frac{\ell}{2}\right)=\frac{2\ell+1}{2N}. 
\]

In summary, Mallows' $C_p$ in \eqref{Cpdef} is simplified to
\begin{align}\label{Cpdefsimple}
	C_p(\ell) = 
	\frac{N}{2}(\|\hat{\vct{c}}_{\ell+1:N}\|_2^2+\hat{c}_{N}^2)+ 2  \hat{\sigma}^2 \left( \ell+1-\frac{2\ell+1}{2N} \right).
\end{align}
This can be computed for all $\ell$ in $O(N)$ operations, as done in the MATLAB code below. 
In practice, since the Chebyshev coefficients decay rapidly enough, the extra term $\hat{c}_{N}^2$ plays little role and can be ignored. 
Since $C_p(\ell)$ can be viewed as an unbiased {estimator} of the predictive mean squared error from Lemma~\ref{lem_Cp} below, we select the polynomial degree $n$ by minimizing $C_p(\ell)$:
\[
n = \argmin_\ell C_p(\ell).
\]
Minimization of $C_p$ has been shown to attain asymptotic minimaxity in function estimation {in the $L^2$ norm} \cite{massart2001gaussian,tsybakov}.

We note that in a general setting of variable selection in linear regression \cite{james2013introduction}, one would need to fit all possible subsets of variables to choose the best one. Fortunately, here in the context of polynomial approximation this is unnecessary---the variables are already 'ordered' in terms of the Chebyshev degree, so we can simply and reliably consider only the subsets of the leading Chebyshev polynomials. This reduces the computational work dramatically: from $O(2^N)$ to $O(N)$. 
Keeping the low-degree Chebyshev polynomials is a natural strategy, as it can be interpreted as keeping the low-frequency terms in a Fourier series (which is very standard practice), based on the link between Chebyshev series and Fourier series~\cite[Ch.~3]{trefethenatap}. Moreover, a similar point is discussed in~\cite[\S~12.3]{draper1998applied} (Retaining terms in polynomial models). After examining concrete examples, they conclude that, if the $k$-th order term is included in the polynomial regression model, then it is natural to include all the lower order terms up to the $(k-1)$-th order.

\subsection{Algorithm \mainalg}
We now present \mainalg\ in Algorithm~\ref{alg:noisy_interp}. The algorithm essentially truncates {a high-degree} Chebyshev interpolant at a degree determined by Mallows' $C_p$. 

\begin{algorithm}[H]
	\caption{\mainalg: Approximates noisy univariate functions on $[-1,1]$.}
	\label{alg:noisy_interp}
	\begin{algorithmic}[1]
		\Require A computational routine to sample $f:[-1,1]\rightarrow \mathbb{R}$ with noise as in~\eqref{eq:yieval}, 
		and $N$: computational budget on the allowed number of samples ($(N+1)$ Chebyshev points are sampled).
		\Ensure A polynomial $p_n\approx f$ of degree $n(<N)$. 
		\State Sample at $(N+1)$ Chebyshev points 
		$\{x_i\}_{i=0}^N$ to obtain 
		$\{y_i\}_{i=0}^N$. 
		\State Chebyshev interpolation: Find the degree $N$ polynomial interpolant $\tilde p_N(x)=\sum_{i=0}^N c_j T_j(x)$ such that $\tilde p_N(x_i)=y_i$ for $i=0,1,\ldots, N$.  
		\State Mallows' $C_p$: Compute $C_p(\ell)$ in \eqref{Cpdefsimple} for $\ell=0,1,\dots,\bar{n}$, and select $n\in\{0,1,\dots,\bar{n}\}$ that minimizes $C_p(n)$.
		\State Truncate the Chebyshev coefficients at degree $n$. That is, output $p_n(x)=\sum_{j=0}^n c_j T_j(x)$.
	\end{algorithmic}
\end{algorithm}

When $f$ is smooth enough (so that $|c_m|\ll \sigma$ for all $m>n$), the overall approximation error is $O(\sigma\sqrt{{n}/{N}})$, as we make precise in Section~\ref{sec:uniformconv}; note that this can be significantly smaller than the noise level $\sigma$, that is, oversampling (as compared to interpolation $n=N$) reduces the effect of noise. 

The algorithm can be executed in $O(N\log N)$ operations using the DCT or FFT~\cite{trefethenatap}. 
Using the Chebfun toolbox~\cite{chebfunofficial} one can execute the core algorithm succinctly in a few lines of MATLAB codes: 

\actionbox{
	{\tt X = chebpts(N{+1});}\\
	{\tt Y = f(X); \hfill \% sample f w/ noise at Chebyshev pts}\\
	{\tt pN = chebfun(Y); \hfill \% Chebyshev interpolation}\\
	{\tt c = chebcoeffs(pN); \hfill \% Chebyshev coefficients of pN}\\
	{\tt n = MallowsCp(c); \hfill \% choose degree using Mallows' Cp}\\
	{\tt p = chebfun(c(1:n+1),'coeffs'); \hfill\% Output truncated Cheb-coeffs}
}

Aside from {\tt f}, which is the application-dependent noisy evaluation routine, all functions are provided by Chebfun, except 
{\tt MallowsCp}, which implements Mallows' $C_p$ as follows.
Note that the index of an array starts from one, not zero, in MATLAB.

\verbatiminput{MallowsCp.m}

\paragraph{Comparison: \mainalg\ vs. unweighted least-squares}
Lemma~\ref{lem:equiv} shows that the output of 
\mainalg\ is the solution of a weighted least-squares problem~\eqref{eq:DTc-f} with weight $\mtx{D}$ and Chebyshev sampling. 
{Given that $\mtx{D}$ is almost equal to $\mtx{I}$ corresponding to the unweighted problem~\eqref{eq:Tc-f}, it} is unsurprising that there is usually little difference in their approximation quality in practice. 
Nonetheless, the two algorithms differ in important ways: computational cost, input requirement, generality, and numerical stability. 
The unweighted least-squares solution requires $O(Nn^2)$ operations. 
This can sometimes be improved to $O(Nn)$ by using an iterative method such as LSQR~\cite{paige1982lsqr}, when $T$ can be shown to be well-conditioned; this includes the case where $\{x_i\}_{i=0}^N$ are Chebyshev points. Even so, \mainalg\  with $O(N\log N)$ operations is faster unless $n=O(\log N)${, and a high degree $n=O(N)$ is recommended when the noise level is low and the function $f$ is not very smooth}.  

An advantage of the unweighted least-squares approach is that 
the sample points $\{x_i\}_{i=0}^N$ need not be Chebyshev points, so it is applicable more generally. When $\{x_i\}_{i=0}^N$ are prescribed and cannot be chosen by the user, unweighted least-squares becomes the recommended approach. In this case, there is a subtle stability issue that one should bear in mind when using unweighted least-squares, related to the growth of the Lebesgue constant.

\section{Variance analysis}\label{sec:var}
In this section, we study the convergence of {\mainalg} by examining the variance of the approximation $p_n$ evaluated at a specific point {$x\in[-1,1]$}\footnote{Technically $|x|>1$ is allowed and the forthcoming analysis holds for any $x$ {except where 
		$|T(x)|\leq 1$ is used}; however, the Chebyshev polynomials and Lebesgue {function} grow rapidly outside $[-1,1]$, and so does the error.}, that is, $\mbox{Var}[p_n(x)] = \mathbb{E}[(p_n(x)-\mathbb{E}[p_n(x)])^2]$.
Regarding the bias, see the discussion at the end of this section.
Throughout, all expectations are taken with respect to the noise random variables $\{\epsilon_i\}_{i=0}^N$. 
Since the outputs of {\mainalg} and the (unweighted) least-squares method  are closely related and almost identical as just discussed, we shall first analyze the least-squares method, which is more general in that the sample points need not be Chebyshev points, and point out special features that arise when specializing to {\mainalg} in Algorithm~\ref{alg:noisy_interp}. 
The variance of the least-squares method has been studied extensively, including Cohen et al~\cite{cohen2013stability} and Dahlquist and Bjork~\cite[\S~4.5.6]{dahlquist2008numerical}. 
We partially rederive them here, for completeness and because we use the arguments to obtain convergence results in the $L_\infty$ norm in the next section.

\subsection{Unweighted least squares}
First, we consider the (unweighted) least squares method.
Let $\vct{y}=[y_0,y_1,\ldots,y_N]^T \in \mathbb{R}^{N+1}$ and $\mtx{V} \in \mathbb{R}^{(N+1) \times (n+1)}$ be the (generalized) Vandermonde matrix given by 
{$V_{i,j}=\phi_{j-1}(x_{i-1})$}. 
The least-squares problem is given by $\minimize_{\vct{c}} \|\mtx{V}\vct{c}-\vct{y}\|_2^2$.
Its solution is $\vct{\hat c} = (\mtx{V}^{T} \mtx{V})^{-1} \mtx{V}^{T} \vct{y}  \in \mathbb{R}^{n+1}$, which provides the approximation \begin{equation}
	\label{eq:px}
	p_n(x)=\sum_{j=0}^n \hat c_{j} \phi_j(x)=\vct{\phi}(x)^T\vct{\hat c} =\vct{s}(x)^T \vct{y}, \end{equation}
where $\vct{\phi}(x) = [\phi_0(x),\phi_1(x),\ldots, \phi_n(x)]^{T}  \in \mathbb{R}^{n+1}$ and $\vct{s}(x)=\mtx{V} (\mtx{V}^T \mtx{V})^{-1} \vct{\phi}(x)$.
Thus, 
\begin{align}
	\mbox{Var}[p_n(x)] =\mbox{Var}[\vct{s}(x)^T\vct{y}]=\vct{s}(x)^T\mbox{Cov}[\vct{y}]\vct{s}(x) =\vct{s}(x)^T\mbox{Cov}[\vct{\epsilon}]\vct{s}(x), \label{eq:varexact}
\end{align}
where $\mbox{Cov}[\vct{\epsilon}]$ is the covariance matrix of $\vct{\epsilon}$, which is a multiple of identity when the noise is uncorrelated and has the same variance. 
Therefore,  
\begin{align}  
	\mbox{Var}[p_n(x)] \leq \| \vct{s}(x) \|_2^2 \| \mbox{Cov}[\vct{\epsilon}] \|_2 \leq \| \mtx{V} (\mtx{V}^T \mtx{V})^{-1} \|_2^2 \| \vct{\phi}(x) \|_2^2 \|\mbox{Cov}[\vct{\epsilon}]\|_2,  \label{eq:vars}
\end{align}
where $\|\vct{\phi}(x)\|_2$ is what is known as the Christoffel function in the theory of orthogonal polynomials, and studied in detail in~\cite{nevai1986geza}, and extensively in the recent paper~\cite{adcock2024optimal} in the context of function approximation via a least-squares approach, with randomly chosen samples from a distribution determined by the Christoffel function. \mainalg\ distinguishes itself from these in that it is a deterministic algorithm and allows for a fast execution.

\subsection{\mainalg}
We now turn to the analysis of \mainalg. 
This boils down to specializing to the case in Lemma~\ref{lem:equiv}, that is, the weighted least-squares problem~\eqref{eq:DTc-f} is solved with $\vct{x}$ set to the Chebyshev points. 
The analysis above carries over verbatim with the substitution $\vct{s}(x)=\mtx{D}^2 \mtx{T} (\mtx{T}^T \mtx{D}^2 \mtx{T})^{-1} \vct{t}(x)$, where $\vct{t}(x)=[T_0(x),T_1(x),\dots,T_n(x)]$. 
Also, we have $\| \mtx{D}^2 \mtx{T} \|_2 \leq \| \mtx{D} \mtx{T} \|_2 = \sqrt{N}$, where we used $\mtx{T}^T \mtx{D}^2 \mtx{T} =\frac{N}{2} \diag (2,1,\dots,1,1)$.
Also, since $|T_j(x)| \leq 1$ for every $j$ and $x \in [-1,1]$, we have $\| \vct{t}(x) \| \leq \sqrt{n+1}$.
Thus,
\begin{align}
	\| \vct{s}(x) \|_2 &\leq \| \mtx{D}^2 \mtx{T} \|_2 \| (\mtx{T}^T \mtx{D}^2 \mtx{T})^{-1} \|_2 \| \vct{t}(x) \|_2 \leq 2 \sqrt{\frac{n+1}{N}}. \label{norm_s}
\end{align}
Therefore, 
\begin{equation}  \label{eq:varsD}
	\mbox{Var}[p_n(x)] = \vct{s}(x)^T \mbox{Cov}[\vct{\epsilon}] \vct{s}(x) \leq \frac{4(n+1)}{N}\|\mbox{Cov}[\vct{\epsilon}]\|_2,
\end{equation}
which is simply $\frac{4(n+1)}{N}\sigma^2$ when $\epsilon_i$  are uncorrelated with variance $\sigma^2$. 
One clearly sees here that the effect of noise can be reduced by sampling more {keeping $n$ fixed}. This is analogous to the convergence in Monte Carlo methods based on the central limit theorem, with convergence speed governed by the inverse square root of the sample size~\cite{cohen2,MCLSpreprint}. 
It is worth noting that {in practical approximation}, one typically truncates a Chebyshev series once convergence is achieved to the order of either noise level or machine precision. The variance result here highlights the fact that by choosing an appropriate $n$ according to convergence of the coefficients up to the noise level \emph{divided by $\sqrt{N/n}$} 
one can get accuracy better than noise level. 
{The fact that the accuracy can be improved by taking more samples is classical, and is the foundational fact behind Monte Carlo methods. 
	For related recent studies where sample points are taken randomly, we refer to \cite{cohen2013stability,cohen2,migliorati2015convergence}. 
}

The mean squared error of \mainalg\ is the sum of the variance and squared bias:
\[
\mathbb{E}[(p_n(x)-f(x))^2] = \mbox{Var}[p_n(x)] + \mbox{Bias}[p_n(x)]^2,
\]
where $\mbox{Bias}[p_n(x)] =\mathbb{E}[p_n(x)]-f(x)$.
Here, we analyze the bias.
Consider the decomposition $f(x)=p^*_n(x)+r_n(x)$, where 
$p^*_{n}(x)$ is the best (minimax) polynomial approximant that minimizes $\|p-f\|_\infty = \sup_{x \in [-1,1]} |p(x)-f(x)|$ over all polynomials $p$ of degree $n$ or less. 
Define $\vct{p}^*_{n}=[p^*_n(x_0),p^*_n(x_1),\dots,p^*_n(x_N)]^T=\T \vct{c}, \vct{r}_n=[r_n(x_0),r_n(x_1),\dots,r_n(x_N)]^T$, and $\vct{f}=\vct{p}^*_{n}+\vct{r}_{n}$.
Then, \begin{align}
	\label{eq:bias}
	\mbox{Bias}[p_n(x)]&=\mathbb{E}[\vct{s}(x)^T (\vct{f}+\vct{\epsilon})]-(\vct{t}(x)^T \vct{c} + r_n(x)) \nonumber \\
	&=\underbrace{\vct{s}(x)^T \vct{p}_n^* - \vct{t}(x)^T \vct{c}}_{=0} + \vct{s}(x)^T \vct{r}_n-r_n(x) \nonumber \\
	&=\vct{s}(x)^T \vct{r}_n-r_n(x).
\end{align}
Here we used the fact 
$\vct{s}(x)^T \vct{p}_n^* = \vct{t}(x)^T \vct{c}$. This holds because $\vct{s}(x)^T \vct{p}$
is the evaluation at $x$ of the least-squares fit to $\vct{p}$, which is trivially $\vct{p}$ itself if $p$ is a degree-$n$ polynomial. On the other hand, $\vct{t}(x)^T \vct{c}$ is by definition $p_n^*(x)$. 

Therefore, using $\| \vct{r}_n \|_2 \leq \sqrt{N+1} \| r_n \|_{\infty}$ and \eqref{norm_s}, we obtain 
\begin{align*}
	| \mbox{Bias}[p_n(x)] | & \leq \| \vct{s}(x) \|_2 \| \vct{r}_n \|_2 + |r_n(x)| \leq \left( 2 \sqrt{n+1} \sqrt{1+\frac{1}{N}} + 1 \right) \| r_n \|_{\infty}.
\end{align*}
Generally, $\|r_n\|_\infty\rightarrow 0$ as $n\rightarrow\infty$ at a spectral rate.
For example, $r_n$ decays exponentially (resp.~algebraically) when $f$ is analytic (resp.~differentiable) \cite[Ch.~7,8]{trefethenatap}.
Therefore, the bias also decays as $n\rightarrow\infty$ at a spectral rate.

\section{Pointwise and uniform concentration}\label{sec:inf}
We continue with the convergence analysis of \mainalg, and now focus on pointwise and uniform convergence of $p_n$ to $f$ (or $p_n^*$). 
We assume that the noise ${\epsilon_i}$ are {either} subgaussian or subexponential. A random variable $Z$ with mean $\mu$ is said to be subgaussian with parameter $\sigma$ if
\begin{align} \label{subgaussian}
	\mathbb{E} [\exp (\lambda (Z - \mu))] \leq \exp \left( \frac{\sigma^2 \lambda^2}{2} \right)
	\quad \mbox{for all}\quad \lambda \in \mathbb{R}.
\end{align}
{A Gaussian random variable is subgaussian where $\sigma^2$ is simpy its variance.}
Also, a random variable $Z$ with mean $\mu$ is said to be subexponential with parameter $(\nu,\alpha)$ if
\begin{align} \label{subexponential}
	\mathbb{E} [\exp (\lambda (Z - \mu))] \leq \exp \left( \frac{\nu^2 \lambda^2}{2} \right)
	\quad\mbox{for}\quad |\lambda| < \frac{1}{\alpha}.
\end{align}
See \cite{wainwright2019high} for properties of subgaussian/subexponential random variables.

\subsection{Pointwise concentration}\label{sec:pointwise}
In this section we take $x\in[-1,1]$ to be fixed, and examine the error $|p_n(x)-f(x)|$. 
From $|p_n(x)-f(x)| = |p_n(x)-p_n^*(x)-r_n(x)| \leq |p_n(x)-p_n^*(x)|+|r_n(x)| \leq |p_n(x)-p_n^*(x)|+\| r_n \|_{\infty}$,
\begin{align}
	\mathbb{P}[|p_n(x)-f(x)|>t+\| r_n \|_{\infty}] \leq \mathbb{P}[|p_n(x)-p_n^*(x)|>t]. \label{pointwise_1}
\end{align}
Also, from the discussion at the end of the previous section,
\begin{align*}
	p_n(x)-p_n^*(x) &= \vct{s}(x)^T (\vct{f}+\vct{\epsilon}) - \vct{t}(x)^T \vct{c} = \vct{s}(x)^T (\vct{r}_n+\vct{\epsilon}),
\end{align*}
where $\vct{s}(x)=\mtx{D}^2 \mtx{T} (\mtx{T}^T \mtx{D}^2 \mtx{T})^{-1} \vct{t}(x)$.
Thus, using \eqref{norm_s} we obtain
\begin{align*}
	|p_n(x)-p_n^*(x)| &\leq \| \vct{s}(x) \|_2 \| \vct{r}_n \|_2 +|\vct{s}(x)^T \vct{\epsilon}| \\
	&\leq 2 \sqrt{\frac{n+1}{N}} \cdot \sqrt{N+1} \| r_n \|_{\infty} +|\vct{s}(x)^T \vct{\epsilon}| \\
	&\leq \sqrt{8(n+1)} \| r_n \|_{\infty} +|\vct{s}(x)^T \vct{\epsilon}|.
\end{align*}
Therefore,
\begin{align}
	\mathbb{P}[|p_n(x)-p_n^*(x)|>t+\sqrt{8(n+1)} \| r_n \|_{\infty}] \leq \mathbb{P}[|\vct{s}(x)^T \vct{\epsilon}|>t]. \label{pointwise_2}
\end{align}
Combining \eqref{pointwise_1} and \eqref{pointwise_2} {yields}
\begin{align}
	\mathbb{P}[|p_n(x)-f(x)|>t+(\sqrt{8(n+1)}+1) \| r_n \|_{\infty}] \leq \mathbb{P}[|\vct{s}(x)^T \vct{\epsilon}|>t]. \label{pointwise_3}
\end{align}
Note that 
to bound the term 
$| \vct{s}(x)^T \vct{r}_n |$
we used Cauchy-Schwarz. 
This is often a significant overestimate as the $(N+1)$-dimensional vectors $\vct{s}(x)$ and $\vct{r}_n$ are unlikely to be nearly parallel. For example if $\vct{r}_n$ took independent subgaussian entries (which is not true), 
the inner product 
would be $O(\frac{1}{\sqrt{N+1}} \| s(x)\|_2 \|r_n\|_2)$~\cite{vershynin2018high}. If this were true for general $r_n$, the term would be $O(\sqrt{\frac{n+1}{N+1}} \| r_n\|_{\infty})$, roughly $\sqrt{N}$ times smaller. We observe that this indeed seems to be the case in the numerical experiments in Section~\ref{sec:exp}.

By bounding the right-hand side of~\eqref{pointwise_3}
with concentration inequalities, we obtain pointwise concentration results for subgaussian and subexponential cases as follows.

\begin{theorem}\label{thm:point}  
	Suppose that the noise $ \{ \epsilon_i \}$ are independent and subgaussian with parameter $\sigma$ (see \eqref{subgaussian}). Then for any fixed $x\in[-1,1]$, 
	\begin{equation*}
		\mathbb{P} \left[ |p_n(x)-f(x)|>2 t \sigma \sqrt{\frac{n+1}{N}}+(\sqrt{8(n+1)} + 1) \|r_n\|_\infty \right]\leq 2 \exp \left( -\frac{t^2}{2} \right). 
	\end{equation*}
\end{theorem}
\begin{proof}  
	Let $\vct{s}(x)=[s_0(x),s_1(x),\dots,s_n(x)]^T$.
	Since $s_i (x) \epsilon_i$ is subgaussian with parameter $s_i(x) \sigma$, their sum $\vct{s}(x)^T \vct{\epsilon}$ is also subgaussian with parameter $\sqrt{\sum_{i=1}^Ns_i(x)^2\sigma^2}= \sigma \|\vct{s}(x)\|_2$.
	Then, using the Hoeffding inequality~\cite[Ch.~2]{wainwright2019high} and \eqref{norm_s}, {we obtain}
	\[
	\mathbb{P}[|\vct{s}(x)^T \vct{\epsilon}|>t] \leq 2 \exp \left( -\frac{t^2}{2 \sigma^2 \| \vct{s}(x) \|_2^2} \right) \leq 2 \exp \left( -\frac{N t^2}{8 (n+1) \sigma^2} \right).
	\]
	This is equivalent to 
	\[
	\mathbb{P} \left[ |\vct{s}(x)^T \vct{\epsilon}|>2 t \sigma \sqrt{\frac{n+1}{N}} \right] \leq 2 \exp \left( -\frac{t^2}{2} \right).
	\]
	Combining this with~\eqref{pointwise_3} completes the proof. 
\end{proof}

\begin{theorem}\label{thm:point2}  
	If the noise $ \{ \epsilon_i \}$ are independent and subexponential with parameter $(\nu,\alpha)$ (see \eqref{subexponential}), then 
	\begin{align*}
		& \mathbb{P} \left[ |p_n(x)-f(x)|>2 t \frac{\nu^2}{\alpha} \sqrt{\frac{n+1}{N}}+(\sqrt{8(n+1)} + 1) \|r_n\|_\infty \right]\leq 2 \exp \left( -\frac{\nu^2}{2 \alpha^2} t^2 \right)
	\end{align*}
	for $0 \leq t \leq t_*$, and
	\begin{equation*}\mathbb{P} \left[ |p_n(x)-f(x)|>2 t \frac{\nu^2}{\alpha} \sqrt{\frac{n+1}{N}}+(\sqrt{8(n+1)} + 1) \|r_n\|_\infty \right]\leq 2 \exp \left( -\frac{\nu^2}{2 \alpha^2} t \right). 
	\end{equation*}
	for $t > t_*$, where
	\[
	t_* = \frac{\|\vct{s}(x)\|_2^2}{2 \max_i |s_i(x)|} \sqrt{\frac{N}{n+1}}.
	\]
\end{theorem}
\begin{proof}  
	Let $\vct{s}(x)=[s_0(x),s_1(x),\dots,s_n(x)]^T$.
	Since $s_i (x) \epsilon_i$ is subexponential with parameter $(s_i(x) \nu,|s_i(x)| \alpha)$, their sum $\vct{s}(x)^T \vct{\epsilon}$ is also subexponential with parameter $(\nu \|\vct{s}(x)\|_2, \alpha \max_i |s_i(x)|)$.
	Then, by using Prop.~2.9 in \cite{wainwright2019high}, \eqref{norm_s} 
	and $\max_i |s_i(x)| \leq \| \vct{s}(x) \|_2$, we obtain 
	\[
	\mathbb{P}[|\vct{s}(x)^T \vct{\epsilon}|>t] \leq 2 \exp \left( -\frac{t^2}{2 \nu^2 \| \vct{s}(x) \|_2^2} \right) \leq 2 \exp \left( -\frac{N t^2}{8 (n+1) \nu^2} \right)
	\]
	for $0 \leq t \leq \nu^2 \|\vct{s}(x)\|_2^2/(\alpha \max_i |s_i(x)|)$, and
	\[
	\mathbb{P}[|\vct{s}(x)^T \vct{\epsilon}|>t] \leq 2 \exp \left( -\frac{t}{2 \alpha \max_i |s_i(x)|} \right) \leq 2 \exp \left( -\frac{\sqrt{N} \cdot t}{4 \sqrt{n+1} \cdot \alpha} \right)
	\]
	for $t > \nu^2 \|\vct{s}(x)\|_2^2/(\alpha \max_i |s_i(x)|)$.
	Thus,
	\[
	\mathbb{P} \left[|\vct{s}(x)^T \vct{\epsilon}|>2 t \frac{\nu^2}{\alpha} \sqrt{\frac{n+1}{N}} \right] \leq 2 \exp \left( -\frac{\nu^2}{2 \alpha^2} t^2 \right)
	\]
	for $0 \leq t \leq t_*$, and
	\[
	\mathbb{P} \left[|\vct{s}(x)^T \vct{\epsilon}|>2 t \frac{\nu^2}{\alpha} \sqrt{\frac{n+1}{N}} \right] \leq 2 \exp \left( -\frac{\nu^2}{2 \alpha^2} t \right)
	\]
	for $t > t_*$.
	Combining this with~\eqref{pointwise_3} completes the proof. 
\end{proof}

\subsection{Uniform concentration}\label{sec:uniformconv} 
Here is our main theoretical result, which bounds the uniform error by $O({\sigma}\sqrt{{n}/{N}}+\sqrt{n}\|r_n\|_\infty)$ with high probability. 
\begin{theorem}\label{thm:infnorm}
	Suppose that the noise $ \{ \epsilon_i \}$ are independent and subgaussian, with parameter $\sigma$ (see \eqref{subgaussian}). Then for every $t>0$ we have 
	\begin{align*}
		&\mathbb{P}\bigg[\| p_n-f \|_\infty>\left( \frac{2}{\pi}\log(n+1)+1 \right) \sqrt{n+1} \left( 2 t \frac{\sigma}{\sqrt{N}}+\sqrt{8} \|r_n\|_\infty \right) + \|r_n\|_\infty \bigg] \nonumber \\
		\leq \ & 2(n+1)\exp \left( -\frac{t^2}{2} \right).   \label{eq:mainthmexp}
	\end{align*}
\end{theorem}
\begin{proof}
	As in Theorem~\ref{thm:point}, we have for any fixed $x\in[-1,1]$
	\begin{equation*}
		\mathbb{P} \left[ |p_n(x)-p_n^*(x)|>2 t \sigma \sqrt{\frac{n+1}{N}}+\sqrt{8(n+1)} \|r_n\|_\infty \right]\leq 2 \exp \left( -\frac{t^2}{2} \right).
	\end{equation*}
	Taking $x$ to be the $n+1$ Chebyshev points $y_i$ (note that these are not the sample points $x_i$, which are $N+1$ Chebyshev points) and using  the union bound, we see that 
	\begin{align}
		& \mathbb{P}\bigg[|p_n(x)-p_n^*(x)|>2 t \sigma \sqrt{\frac{n+1}{N}}+\sqrt{8(n+1)} \|r_n\|_\infty \ \mathrm{for\ some\ } x\in\{y_0,\ldots,y_{n} \}\bigg] \nonumber \\
		\leq \ & 2(n+1)\exp \left( -\frac{t^2}{2} \right). \label{ineq1}
	\end{align}
	
	Let $\mathcal{L}_n$ be the interpolation operator at the $n+1$ Chebyshev points. 
	The Lebesgue constant is defined by
	\begin{align}  \label{eq:deflebesgue}
		\|\mathcal{L}_n\|:=\sup_{g\in C[-1,1]}\frac{\|\mathcal{L}_ng\|_\infty}{\|g\|_\infty}.  
	\end{align}
	{We} have the sharp estimates~\cite[Ch.~15]{trefethenatap} 
	\[
	\frac{2}{\pi}\log(n+1)+0.52 < \|\mathcal{L}_n\|\leq \frac{2}{\pi}\log(n+1)+1. 
	\]
	We shall take a particular case of 
	$g$ in~\eqref{eq:deflebesgue} as follows. 
	Let $g(y_i)=p_n(y_i)-p_{n}^*(y_i)$ at the $n+1$ Chebyshev points, and let $g$ be such that $\|g\|_\infty=\max_i|p_n(y_i)-p_n^*(y_i)|$ (such a continuous function obviously exists; for example $g$ can be piecewise linear). 
	We have 
	\[ 
	\|\mathcal{L}_ng\|_\infty\leq \|\mathcal{L}_n\|\|g\|_\infty.  
	\]
	Now note that $\mathcal{L}_ng=\mathcal{L}_n(p_n-p_n^*)=p_n-p_n^*$, because the operator $\mathcal{L}_n$ depends only on the values at the Chebyshev points, and interpolation of a degree-$n$ polynomial at $n+1$ points yields the same polynomial. 
	Putting these together, we obtain  
	\begin{align} 
		\|p_n-p_n^*\|_\infty \leq \left( \frac{2}{\pi}\log(n+1)+1 \right) \max_i |p_n(y_i)-p_n^*(y_i)|.  \label{ineq2}
	\end{align}
	
	Combining \eqref{ineq1} and \eqref{ineq2} yields
	\begin{align*} 
		&\mathbb{P}\bigg[\| p_n-p_n^* \|>\left( \frac{2}{\pi}\log(n+1)+1 \right) \sqrt{n+1} \left( 2t \frac{\sigma}{\sqrt{N}}+\sqrt{8} \|r_n\|_\infty \right) \bigg] \\
		\leq \ & 2(n+1)\exp \left( -\frac{t^2}{2} \right).
	\end{align*}
	Since $\| p_n-f \|_\infty \leq \| p_n-p_n^* \|_\infty+\| r_n \|_\infty$, {the proof is completed}.
\end{proof}

\begin{theorem}\label{thm:infnorm2}
	If the noise $ \{ \epsilon_i \}$ are independent and subexponential with parameter $(\nu,\alpha)$ (see \eqref{subexponential}), then \begin{align*}
		&\mathbb{P}\bigg[\| p_n-f \|_{\infty}>\left( \frac{2}{\pi}\log(n+1)+1 \right) \sqrt{n+1} \left( 2 t \frac{\nu^2}{\alpha} \frac{1}{\sqrt{N}}+\sqrt{8} \|r_n\|_\infty \right) + \|r_n\|_\infty \bigg] \nonumber \\
		\leq \ & 2(n+1)\exp \left( -\frac{\nu^2}{2 \alpha^2} t^2 \right).
	\end{align*}
	for $0 \leq t \leq t_*$, and
	\begin{align*}
		& \mathbb{P} \left[ \| p_n-f \|_{\infty} >\left( \frac{2}{\pi}\log(n+1)+1 \right) \sqrt{n+1} \left( 2 t \frac{\nu^2}{\alpha} \frac{1}{\sqrt{N}}+\sqrt{8} \|r_n\|_\infty \right) + \|r_n\|_\infty \right] \\
		\leq \ & 2(n+1) \exp \left( -\frac{\nu^2}{2 \alpha^2} t \right). 
	\end{align*}
	for $t > t_*$, where
	\[
	t_* = \frac{\|\vct{s}(x)\|_2^2}{2 \max_i |s_i(x)|} \sqrt{\frac{N}{n+1}}.
	\]
\end{theorem}
\begin{proof}
	The proof is essentially the same {as} Theorem~\ref{thm:infnorm}, {except that} we use Theorem~\ref{thm:point2} instead of Theorem~\ref{thm:point}. \end{proof}

From Theorem~\ref{thm:infnorm}, the infinity-norm error $\|p_n-f\|_\infty$ is $O(\sigma \sqrt{n/N} + \sqrt{n} \| r_n \|_{\infty})$ with high probability.
Recall that $\| r_n \|_{\infty} =\|f-p_n^*\|_\infty \to 0$ as $n \to 0$ at a spectral rate (exponentially for analytic funtions, algebraically for differentiable functions) \cite[Ch.~7,8]{trefethenatap}, and \mainalg\ selects $n$ from $1,2,\dots,\bar{n}(=\lfloor (N+1)/2\rfloor)$ by minimizing Mallows' $C_p$.
Thus, the convergence $p_n\rightarrow f$ of \mainalg\ as $N\rightarrow\infty$ is expected to exhibit the following two-stage behavior:
\begin{itemize}
	\item When $N$ is small so that $\| r_n \|_{\infty}$ cannot be made sufficiently small, $n \approx \bar{n}$ is selected and $\sigma \sqrt{n/N} < \sqrt{n} \| r_n \|_{\infty}$. Thus, $\|p_n-f\|_\infty$ is $O(\sqrt{n} \|r_n\|_\infty)$ with high probability and decays at a spectral rate. 
	
	\item When $N$ is large enough, Mallows' $C_p$ selects an appropriate degree $n$ such that $\sigma \sqrt{n/N} \approx \sqrt{n} \| r_n \|_{\infty}$. Thus, $\|p_n-f\|_\infty$ is bounded by a modest multiple of $\sigma \sqrt{n/N}$ with high probability and decays at Monte-Carlo rate $O(1/\sqrt{N})$ approximately.
\end{itemize}
The interpretation for Theorem \ref{thm:infnorm2} is similar.
We will confirm this with numerical experiments in Section~\ref{sec:exp}. 
Providing a theoretical guarantee for the convergence with Mallow's $C_p$ 
is a compelling open problem. Recent findings on consistency of Mallows' $C_p$ in high-dimensional settings \cite{fujikoshi2014consistency,yanagihara2015consistency} may be useful.

Above, we derived non-asymptotic concentration results for fixed $N$ and $n$.
We reiterate that our analysis does not  account for the effect of selecting $n$ by minimizing Mallows' $C_p$ in \mainalg, although the numerical results below are well explained by the theory presented above. 
It is an interesting future work to develop a rigorous bound for the concentration behavior of \mainalg, including Mallows' $C_p$.

\section{Numerical experiments}\label{sec:exp}
We present a number of experiments to illustrate the performance and properties of \mainalg. 

\subsection{Degree selection}
In addition to {Mallows' $C_p$ and related tools} in statistics, 
there are approaches for degree selection proposed in numerical analysis, which are largely based on examining the Chebyshev coefficients. 
One commonly-used method is Chebfun's StandardChop~\cite{aurentz2017chopping}, which employs a sophisticated and somewhat complicated algorithm to determine the degree. 
{When dealing with noisy functions, it requires} the user to input the noise level. 

Here we compare degree selection by Mallows' $C_p$ with Chebfun's StandardChop. We use StandardChop with two input noise levels $\hat\sigma$: one with the (unknown) exact value $\hat\sigma=\sigma$ (shown as Chop in the figures), and one sets to $\hat\sigma=\mbox{median}(\vct{c}(N/2:end))$, which accounts for the $1/\sqrt{N}$ error reduction, i.e., it takes the input noise $\hat\sigma\approx \sigma/\sqrt{N}$ (Chop-reduced). 

\paragraph{Runge function}
We first take the Runge function $f(x)=\frac{1}{25x^2+1}$ as we did in the introduction. 
This function is smooth and analytic, and the Chebyshev coefficients $|c_j|$ decay geometrically in $j$. We vary the noise level $\sigma$ from $10^{-1}$ to $10^{-8}$ and illustrate the selected degree in Figure~\ref{fig:rungedeg}; the marks indicate the final coefficient of the truncated polynomial {for each method}. 

\begin{figure}[htpb]
	\begin{minipage}[t]{0.325\hsize}
		\includegraphics[width=1\textwidth]{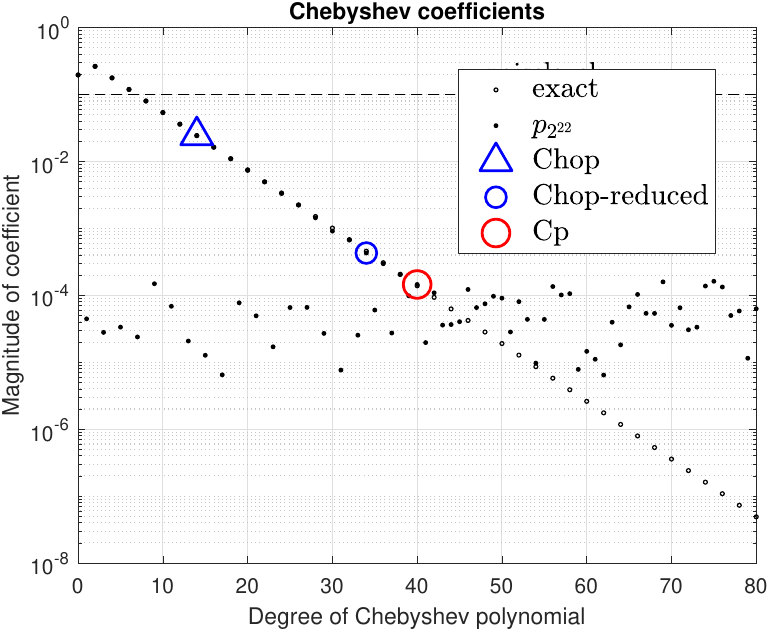}
	\end{minipage}   
	\begin{minipage}[t]{0.325\hsize}
		\includegraphics[width=1\textwidth]{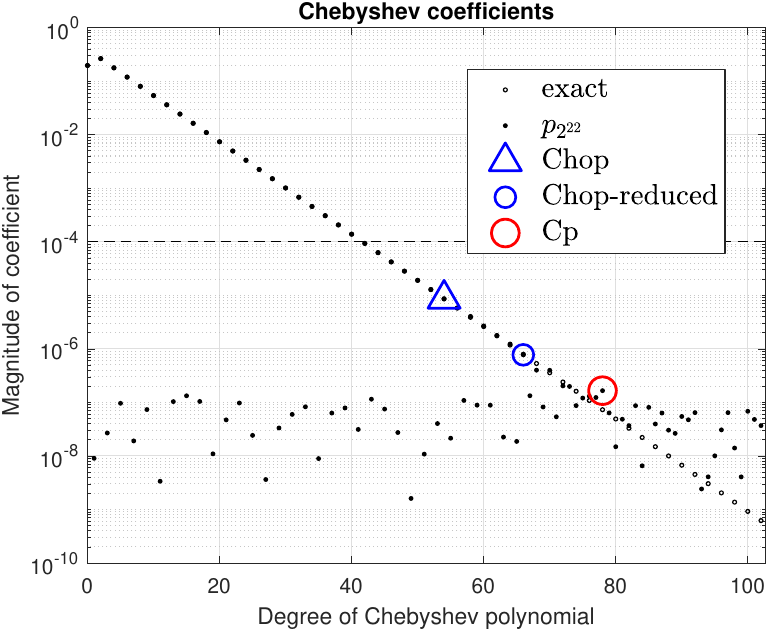}
	\end{minipage}
	\begin{minipage}[t]{0.325\hsize}
		\includegraphics[width=1\textwidth]{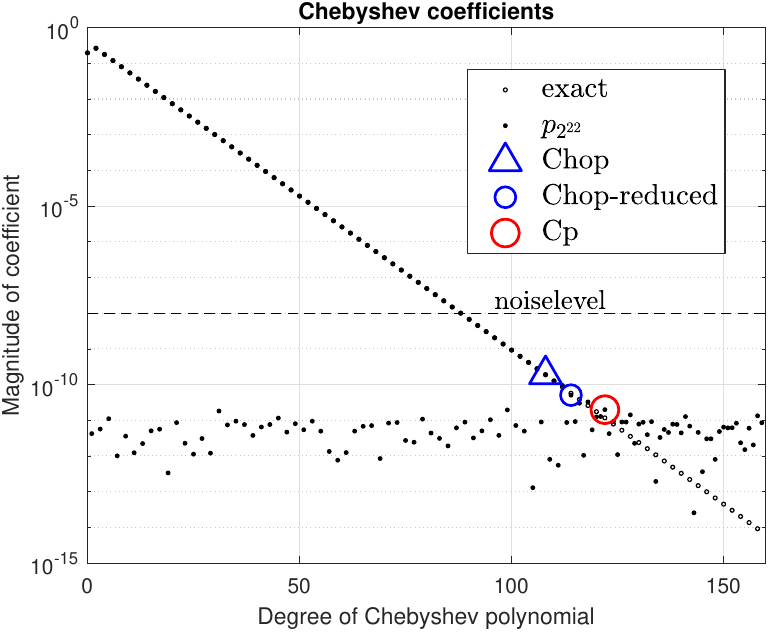}
	\end{minipage}
	\caption{Degree selection for noisy Runge (smooth analytic) function $f(x)=\frac{1}{25x^2+1}$, with varying noise levels $\sigma=10^{-1}$ (left), $\sigma=10^{-4}$ (center), and $\sigma=10^{-8}$ (right).}
	\label{fig:rungedeg}
\end{figure}

We see that the methods work reasonably well in most cases, adapting to the noise level and the number of samples. 
StandardChop usually gives the smallest degree, which is often somewhat premature{, as it does not account for the $\sqrt{n/N}$ noise reduction effect}. 

\paragraph{Highly noisy case}
We next set the noise level to $\sigma=10$ and report in Figure~\ref{fig:runge2}. 
We find it satisfactory to see (in the left panel) that even with a large noise level such as $\sigma=10$ (so each evaluation is dominated by noise), decent accuracy can be obtained by taking a large $N=2^{22}$. In this example, StandardChop with the reduced noise level gave an enormous output $\approx N$, while StandardChop with the original noise level chooses degree $0$ and Mallows' $C_p$ selected a sensible degree $22$. 

\begin{figure}[htbp]
	\centering
	\begin{minipage}[t]{0.325\hsize}
		\includegraphics[width=1\textwidth]{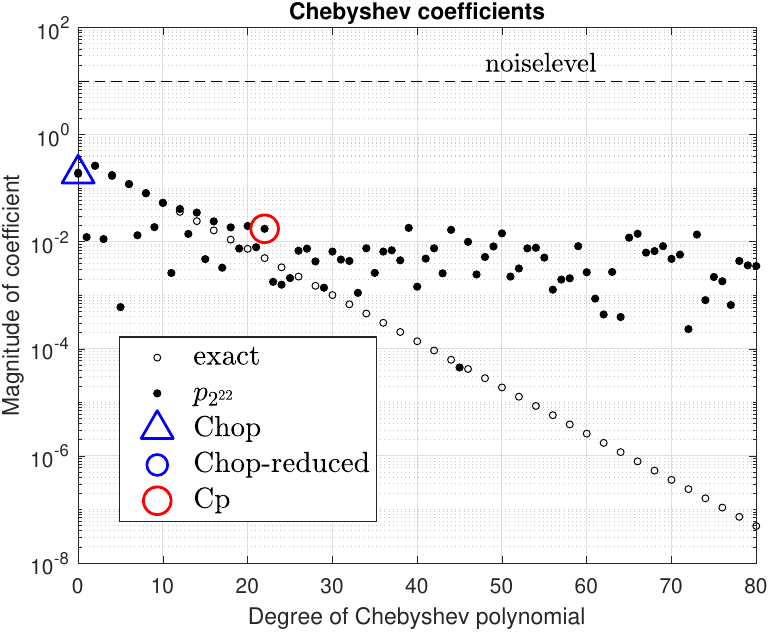}
	\end{minipage}   
	\begin{minipage}[t]{0.325\hsize}
		\includegraphics[width=1\textwidth]{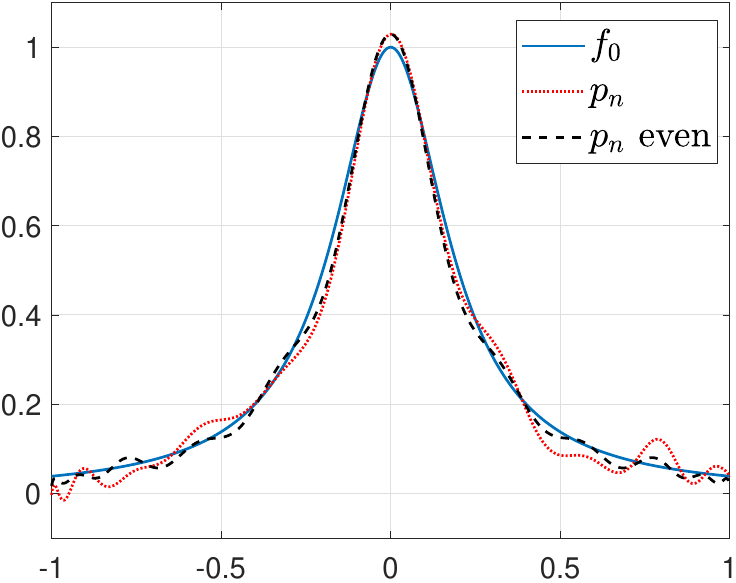}
	\end{minipage}
	\begin{minipage}[t]{0.325\hsize}
		\includegraphics[width=1\textwidth]{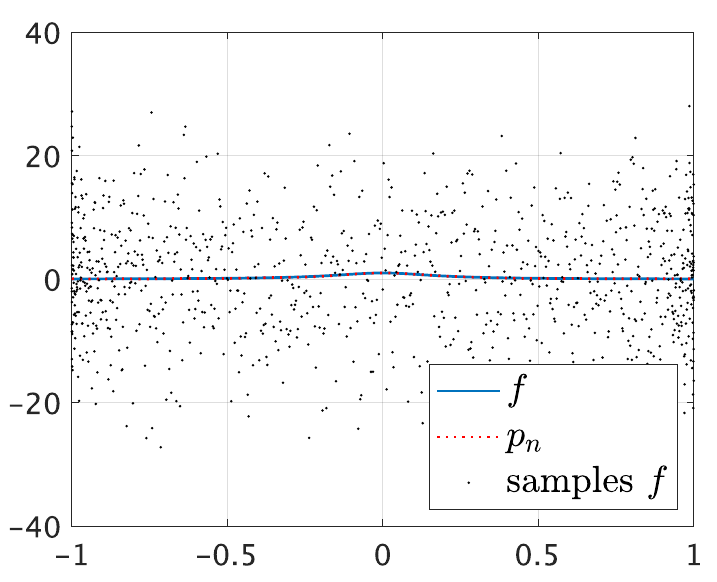}
	\end{minipage}
	
	\caption{Very noisy example $\sigma=10$. Left: same as Figure~\ref{fig:rungedeg} with $\sigma=10$. Center: Target function $f$ (Runge function) and the output of \mainalg\ $p_n$ with $n=22$ from $N=2^{22}$ noisy samples. 
		We also plot the polynomial obtained by setting the odd-degree coefficients of $p_n$ to zero; recall~\cref{fnlabel}. Right: The functions together with the evaluation points; as showing $2^{22}$ points is not feasible, we took a subset of 1000 points chosen uniformly. Even with such noise-dominated evaluations, one obtains a reasonable approximation with \mainalg. 
	}
	\label{fig:runge2}
\end{figure}

\paragraph{Non-analytic functions}
Figure~\ref{fig:abs3deg} repeats the experiments with a less smooth, non-analytic function $f=|x|^3$, for which the Chebyshev coefficients decay algebraically rather than geometrically. 
\begin{figure}[htpb]
	\begin{minipage}[t]{0.325\hsize}
		\includegraphics[width=1\textwidth]{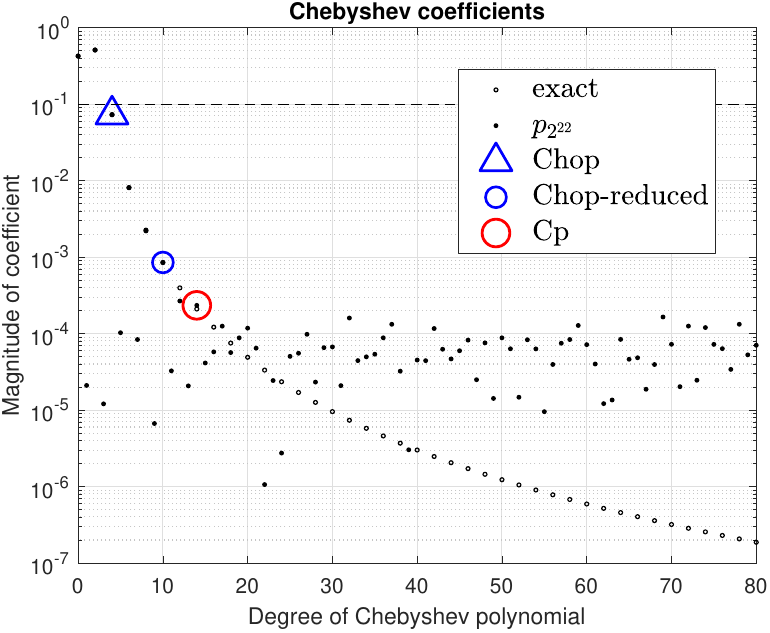}
	\end{minipage}   
	\begin{minipage}[t]{0.325\hsize}
		\includegraphics[width=1\textwidth]{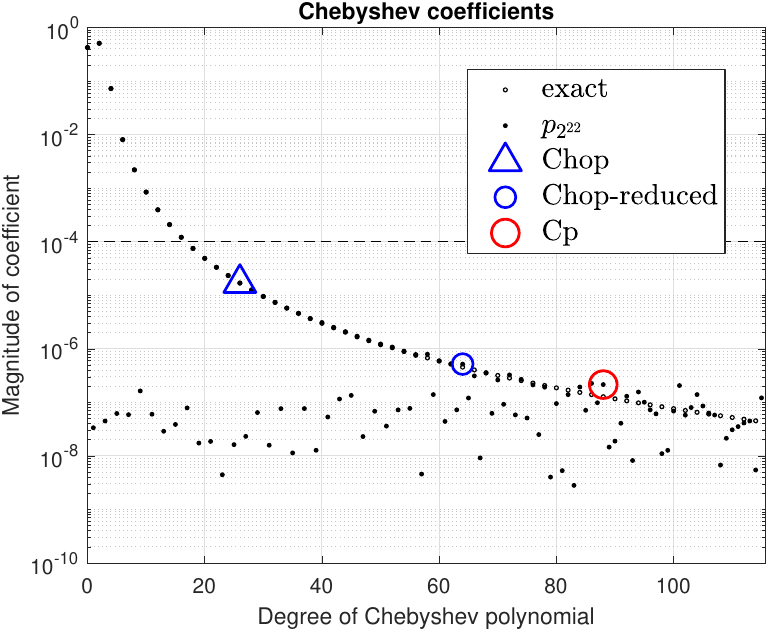}
	\end{minipage}
	\begin{minipage}[t]{0.325\hsize}
		\includegraphics[width=1\textwidth]{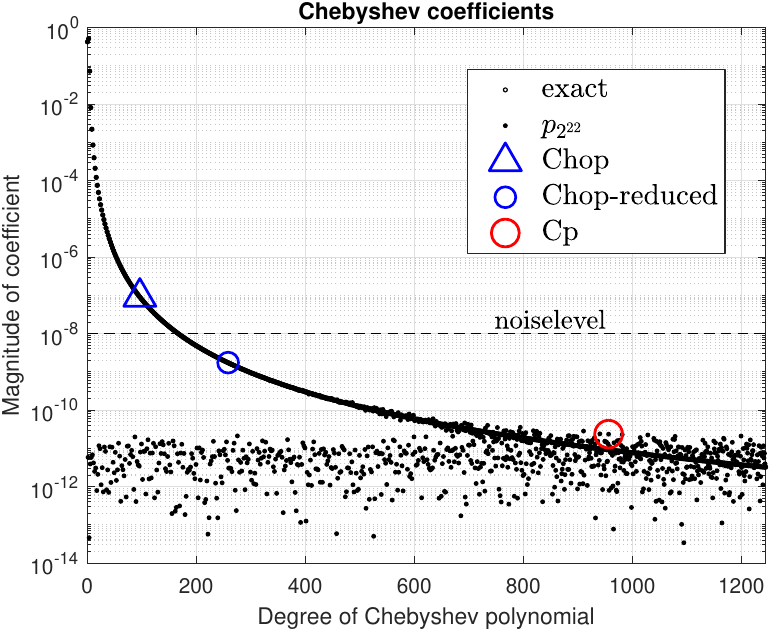}
	\end{minipage}
	\caption{Repeats Figure~\ref{fig:rungedeg} for a different (less smooth) 
		function $f(x)=|x|^3$.}
	\label{fig:abs3deg}
\end{figure}
Largely the same qualitative observations hold here. 
We note that as $n$ becomes quite large here, executing these using the LS approach would be significantly slower than interpolation. 

In all examples so far, we see that Mallows' $C_p$ gives a sensible output. 
This can be understood from the asymptotic minimaxity of the $C_p$-based function estimation \cite{massart2001gaussian}.
By contrast, Chebfun's StandardChop usually chooses a lower degree when the noise input is $\sigma$, and when it is $\sigma\sqrt{n/N}$, 
occasionally selected a degree that is almost equal to $N$. 
Thus, for degree selection in approximating noisy functions, Mallows' $C_p$ is more reliable than StandardChop. This is unsurprising given that StandardChop is not designed specifically to deal with noisy functions but rather to detect convergence plateaus in (usually noiseless) Chebyshev series, whereas the statistical methods are specifically targeted to deal with noise. 

We have repeated these experiments with other types of noise distribution, including uniform and Laplace, and observed that Mallows' $C_p$ always reliably selected a sensible degree $n$. In particular, Figure~\ref{fig:coeffs0} looks almost identical in all these cases (not shown); unsurprisingly the Cauchy distribution caused the error reduction effect to disappear; surprisingly, the degree detection nonetheless performed reasonably well.

\subsection{Concentration behavior: spectral convergence followed by noise reduction}\label{ex:spectral}
Here we illustrate the implications of Theorem~\ref{thm:infnorm}. First, we examine the convergence of \mainalg\ as the number of samples $N$ is increased. We expect the degree $n$ chosen to increase accordingly, and as suggested by Theorem~\ref{thm:infnorm}, we expect the convergence to be spectral until the error reaches $O(\sigma\sqrt{n/N})$. This is confirmed in Figure~\ref{fig:spectral}, where we examine two functions (one smooth+analytic, and one with a discontinuous derivative). The figure clearly demonstrates that when the function is smooth and the noise level is low, rapid convergence is obtained by \mainalg, close to that of the best polynomial approximant.  

\begin{figure}[htbp]
	\centering
	\begin{minipage}[t]{.495\linewidth}
		\includegraphics[width=.915\textwidth]{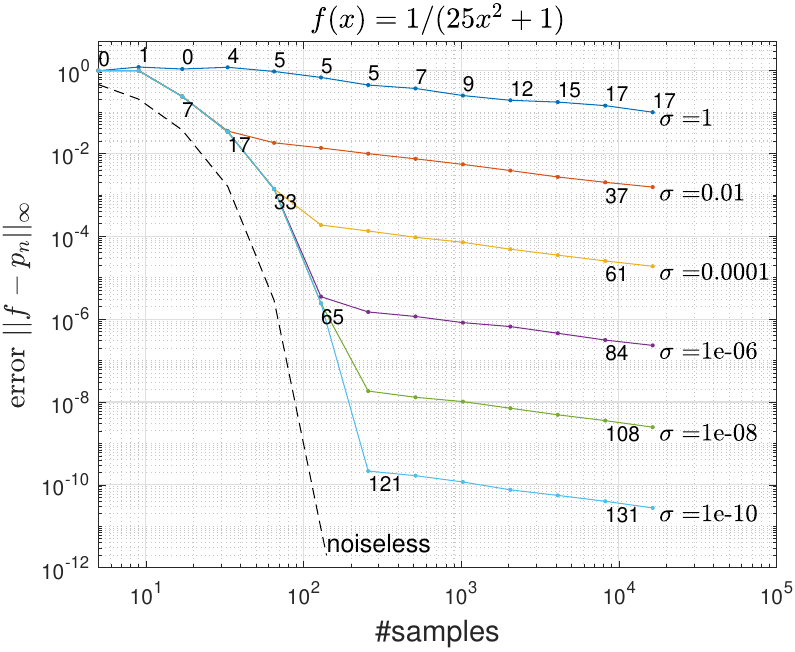}
	\end{minipage}
	\begin{minipage}[t]{.495\linewidth}
		\includegraphics[width=.91\textwidth]{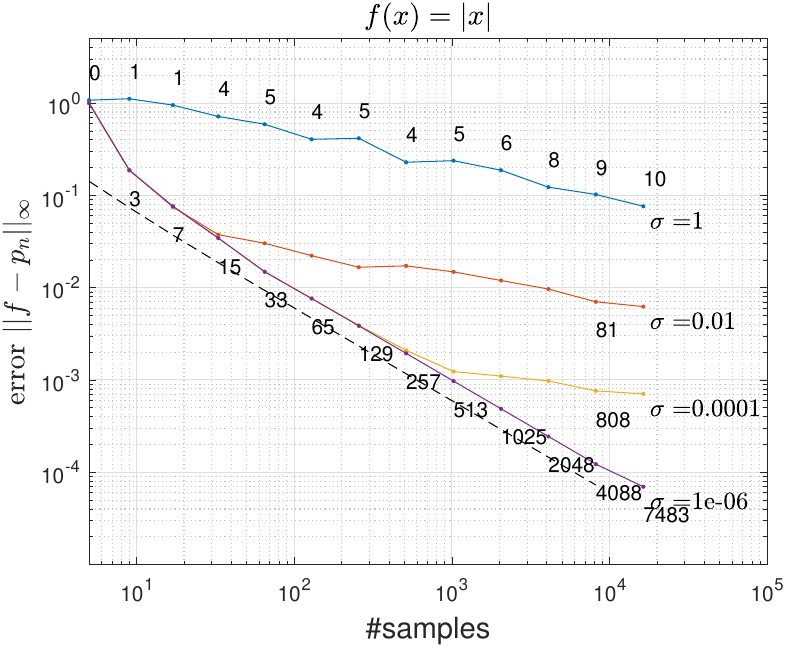}
	\end{minipage}
	\caption{Illustration of spectral convergence followed by $O(\sigma/\sqrt{N})$ error. 
		Left: $f(x)=1/(25x^2+1)$, right: $f(x)=|x|$. 
		The plots show \mainalg's convergence of $\|f-p_n\|_\infty$ as $N$ is varied, for different noise levels $\sigma$. 
		Each data point shows the average value of $\|f-p_n\|_\infty$ out of 10 independent runs. For selected data points, the average degree (rounded) is printed. In dashed lines we also plot the error of Chebyshev interpolation when there is no noise (which is approximately the error of the best polynomial approximation). The plots show that the convergence of \mainalg\ is spectral (and close to the best possible) until the error reaches noise level, after which the $\sigma\sqrt{n/N}$ term dominates. 
	}
	\label{fig:spectral}
\end{figure}

We next illustrate the error concentration phenomenon suggested by Theorem~\ref{thm:infnorm}. We run 1000 independent instances of \mainalg\ applied to the Runge function with noise level $\sigma=10^{-3}$ and $N=2^{13}=8192$. In Figure~\ref{fig:concent-hist} we report the histogram of the values of the resulting error $\|f-p_n\|_{\infty}$, along with the two instances with the largest and smallest errors, and the histogram of the chosen degree $n$. In the left panel, we also plot the 
value 
$\left( \frac{2}{\pi}\log(n+1)+1 \right) \sqrt{n+1} \left( 2  \frac{\sigma}{\sqrt{N}}+\sqrt{8} \|r_n\|_\infty \right) + \|r_n\|_\infty$ (shown as the dashed line in the left panel of Figure~\ref{fig:concent-hist}, with the rounded mean value $n=49$), 
which is what the theorem estimates/bounds $\|p_n-f\|_\infty$ to be in an 'average' case (we took $t=1$ in the theorem, and divided the term with $\sqrt{8}\|r_n\|_{\infty}$ by $\sqrt{N}$ in view of the discussion before Theorem~\ref{thm:point}. Without this division, the estimate would be larger by a factor $\approx 8$ and significantly overestimates the actual errors). We see that the error exhibits strong concentration near its mean, and that the estimate, while being a 
noticeable overestimate (which is expected of such analysis with many opportunities for loose bounds), predicts the right order of magnitude; this is a general phenomenon observed in all examples we tested. The degree is also quite concentrated, here around 50. 
\begin{figure}[htbp]
	\centering
	\begin{minipage}[t]{.45\linewidth}
		\includegraphics[width=1.\textwidth]{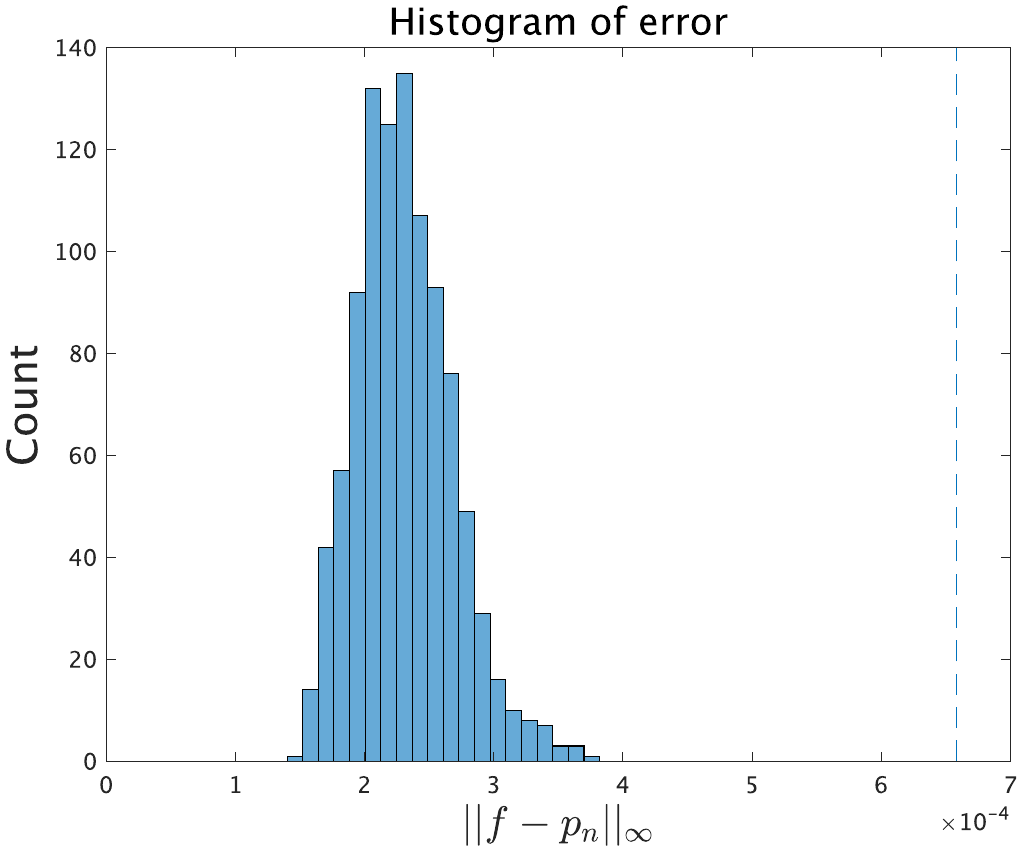}
	\end{minipage}
	\begin{minipage}[t]{.45\linewidth}
		\includegraphics[width=\textwidth]{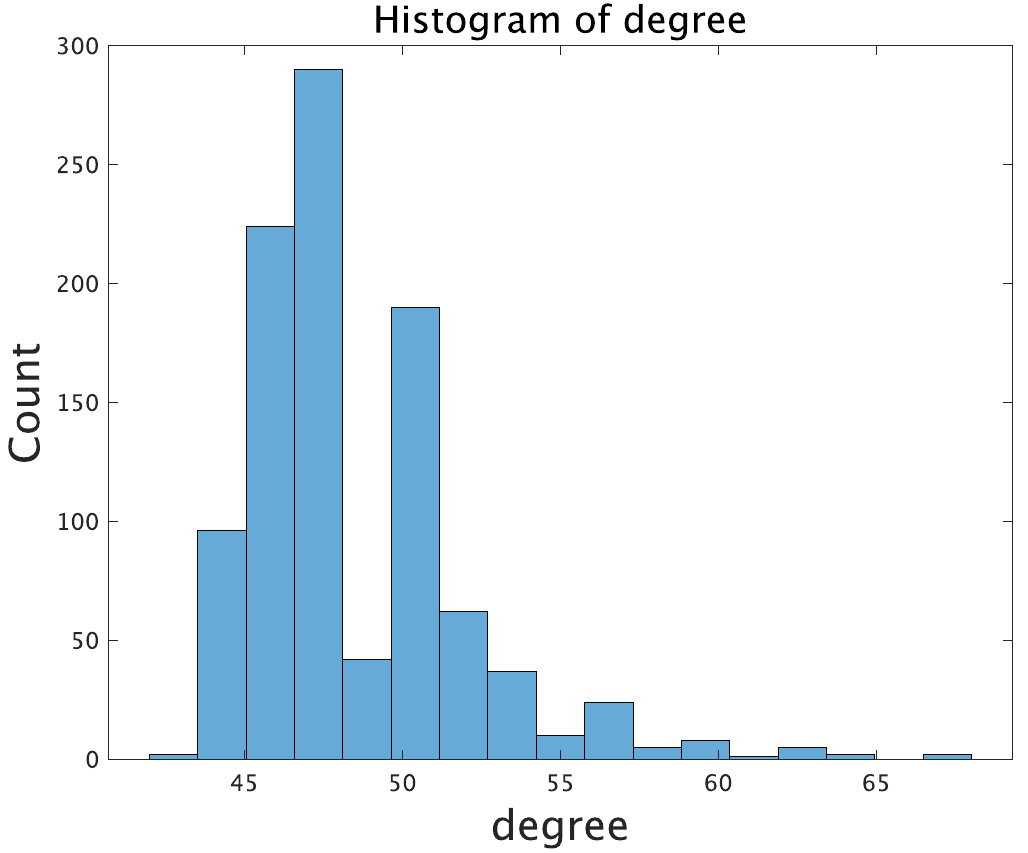}
	\end{minipage}
	\caption{Histogram (out of 1000 runs) of the error (left) and degree (right), and the plot of error $f-p_n$ for the largest (blue) and smallest (red) instance of $\|f-p_n\|_\infty$ (center). }
	\label{fig:concent-hist}
\end{figure}

\section{Discussion} 
Natural extensions of this work would include noisy approximation on other domains (e.g. unions of intervals) and in higher dimensions. 
In hypercubes in $d$ dimensions, a simple idea is to use a tensor-product Chebyshev grid. This will suffer from the curse of dimensionality as $d$ grows. 
To tackle this, possible approaches include utilizing sparse grids~\cite{bungartz2004sparse}, or low-rank methods such as ACA-based methods~\cite{Chebfun2} or tensor-based approaches~\cite{strossner2024approximation}. When working with irregular domains, one will likely need to replace Chebyshev polynomials with a frame~\cite{adcock2019frames,adcock2020approximating} or a better basis~\cite{zhu2023convergencenearoptimalsamplingmultivariate}. 

We have mostly assumed the noise is iid. When noise is heteroskedastic, \mainalg\ will have error uniformly large, roughly proportional to the largest noise divided by $\sqrt{N}$. In such cases, a weighted least-squares approach that respects the heteroskedasticity would be a natural approach. It is an open problem to see if a fast $O(N\log N)$ method is possible in such cases. 

Approximating the derivative(s) in the noisy setting is an important task in many applications~\cite{fioretti1989accurate,ragozin1983error,schafer2011savitzky,zhukweighted}. We expect \mainalg\ to be competitive for this task; a careful comparison and analysis are left for future work.

Another important direction is to extend such methods to nonlinear approximation, in particular rational approximation. 
Rational functions can outperform polynomials when dealing with functions with nonsmoothness or singularities~\cite{trefethenatap}, and algorithms are now available~\cite{nakatsukasa2018aaa} to potentially replace polynomials in many computational settings. Some attempts have been made to tackle noise~\cite{wilber2022data}, but rational functions remain far less extensively explored than polynomials both in theory and algorithmically. 

In this paper, we have focused on the case where the sample points $x_i$ can be chosen to be Chebyshev points. In some cases this is not possible, and one needs to find an approximant from samples at prescribed points, for example equispaced points on $[-1,1]$.  In this case, choosing $n \leq \sqrt{N}$ suffices to ensure numerical stability and this is reminiscent of the conditions for consistency in high-dimensional statistics \cite{portnoy1984asymptotic,yohai1979asymptotic}. There may well be deeper connections between these bodies of literature; working them out is left for future work.

\subsection*{Acknowledgments}
We thank Nick Trefethen for the insightful discussions and comments on a draft, Felix Benning for his comment on DCT-2, and Yifu Zhang for pointing out several typos in an early manuscript. 
We thank the referees for helpful comments.
TM was supported by JSPS KAKENHI Grant Numbers 22K17865, 24K02951 and JST Moonshot Grant Number JPMJMS2024.
YN was supported by EPSRC grants EP/Y010086/1 and EP/Y030990/1.

\bibliographystyle{abbrv}
\bibliography{bib2}

\def\noopsort#1{}\def\l{\char32l}\def\v#1{{\accent20 #1}}
  \let\^^_=\v\def\hbk{hardback}\def\pbk{paperback}
\begin{thebibliography}{10}

\bibitem{adcock2024optimal}
B.~Adcock.
\newblock Optimal sampling for least-squares approximation.
\newblock {\em arXiv preprint arXiv:2409.02342}, 2024.

\bibitem{adcock2019frames}
B.~Adcock and D.~Huybrechs.
\newblock Frames and numerical approximation.
\newblock {\em SIAM Rev.}, 61(3):443--473, 2019.

\bibitem{adcock2020approximating}
B.~Adcock and D.~Huybrechs.
\newblock Approximating smooth, multivariate functions on irregular domains.
\newblock In {\em Forum of Mathematics, Sigma}, volume~8, page e26. Cambridge
  University Press, 2020.

\bibitem{akaike1998information}
H.~Akaike.
\newblock Information theory and an extension of the maximum likelihood
  principle.
\newblock In {\em Selected papers of Hirotugu Akaike}, pages 199--213.
  Springer, New York, 1998.

\bibitem{aurentz2017chopping}
J.~L. Aurentz and L.~N. Trefethen.
\newblock Chopping a {C}hebyshev series.
\newblock {\em ACM Trans. Math. Soft.}, 43(4):33, 2017.

\bibitem{belloni2015some}
A.~Belloni, V.~Chernozhukov, D.~Chetverikov, and K.~Kato.
\newblock Some new asymptotic theory for least squares series: Pointwise and
  uniform results.
\newblock {\em Journal of Econometrics}, 186(2):345--366, 2015.

\bibitem{bishop2006pattern}
C.~M. Bishop and N.~M. Nasrabadi.
\newblock {\em Pattern Recognition and Machine Learning}.
\newblock Springer, New York, 2006.

\bibitem{bungartz2004sparse}
H.-J. Bungartz and M.~Griebel.
\newblock Sparse grids.
\newblock {\em Acta Numer.}, 13:147--269, 2004.

\bibitem{cohen2013stability}
A.~Cohen, M.~A. Davenport, and D.~Leviatan.
\newblock On the stability and accuracy of least squares approximations.
\newblock {\em Found. Comput. Math.}, 13(5):819--834, 2013.

\bibitem{cohen2}
A.~Cohen and G.~Migliorati.
\newblock Optimal weighted least-squares methods.
\newblock {\em SMAI-Journal of Computational Mathematics}, 3:181--203, 2017.

\bibitem{dahlquist2008numerical}
G.~Dahlquist and {\AA}.~Bj{\"o}rck.
\newblock {\em {Numerical Methods in Scientific Computing, Volume I}}.
\newblock SIAM, Philadelphia, 2008.

\bibitem{draper1998applied}
N.~R. Draper and H.~Smith.
\newblock {\em Applied regression analysis}, volume 326.
\newblock John Wiley \& Sons, 1998.

\bibitem{chebfunofficial}
T.~A. Driscoll, N.~Hale, and L.~N. Trefethen.
\newblock {\em Chebfun Guide}.
\newblock Pafnuty Publications, Oxford, 2014.

\bibitem{fioretti1989accurate}
S.~Fioretti and L.~Jetto.
\newblock Accurate derivative estimation from noisy data: a state-space
  approach.
\newblock {\em International Journal of Systems Science}, 20(1):33--53, 1989.

\bibitem{fujikoshi2014consistency}
Y.~Fujikoshi, T.~Sakurai, and H.~Yanagihara.
\newblock Consistency of high-dimensional aic-type and cp-type criteria in
  multivariate linear regression.
\newblock {\em Journal of Multivariate Analysis}, 123:184--200, 2014.

\bibitem{hastie2015statistical}
T.~Hastie, R.~Tibshirani, and M.~Wainwright.
\newblock {\em Statistical learning with sparsity}.
\newblock CRC Press, New York, 2015.

\bibitem{james2013introduction}
G.~James, D.~Witten, T.~Hastie, and R.~Tibshirani.
\newblock {\em An Introduction to Statistical Learning}.
\newblock Springer, New York, 2013.

\bibitem{konishi2008information}
S.~Konishi and G.~Kitagawa.
\newblock {\em Information Criteria and Statistical Modeling}.
\newblock Springer, New York, 2008.

\bibitem{mallows2000some}
C.~L. Mallows.
\newblock Some comments on $c_p$.
\newblock {\em Technometrics}, 15(4):661--675, 1973.

\bibitem{mason2010chebyshev}
J.~C. Mason and D.~C. Handscomb.
\newblock {\em Chebyshev Polynomials}.
\newblock CRC Press, Boca Raton, 2010.

\bibitem{massart2001gaussian}
P.~Massart and L.~Birg{\'e}.
\newblock Gaussian model selection.
\newblock {\em Journal of the European Mathematical Society}, 3(3):203--268,
  2001.

\bibitem{migliorati2015convergence}
G.~Migliorati, F.~Nobile, and R.~Tempone.
\newblock Convergence estimates in probability and in expectation for discrete
  least squares with noisy evaluations at random points.
\newblock {\em J. Multivar. Anal.}, 142:167--182, 2015.

\bibitem{MCLSpreprint}
Y.~Nakatsukasa.
\newblock {Approximate and integrate: Variance reduction in Monte Carlo
  integration via function approximation}.
\newblock {\em arXiv:1806.05492}, 2018.

\bibitem{nakatsukasa2018aaa}
Y.~Nakatsukasa, O.~S{\`e}te, and L.~N. Trefethen.
\newblock The {AAA} algorithm for rational approximation.
\newblock {\em SIAM J. Sci. Comp.}, 40(3):A1494--A1522, 2018.

\bibitem{nevai1986geza}
P.~Nevai.
\newblock G{\'e}za {Freud}, orthogonal polynomials and {C}hristoffel functions.
  {A} case study.
\newblock {\em J. Approx. Theory}, 48(1):3--167, 1986.

\bibitem{paige1982lsqr}
C.~C. Paige and M.~A. Saunders.
\newblock {LSQR}: An algorithm for sparse linear equations and sparse least
  squares.
\newblock {\em ACM Trans. Math. Soft.}, 8(1):43--71, 1982.

\bibitem{portnoy1984asymptotic}
S.~Portnoy.
\newblock Asymptotic behavior of m-estimators of p regression parameters when
  $p^2/n$ is large. i. consistency.
\newblock {\em The Annals of Statistics}, pages 1298--1309, 1984.

\bibitem{ragozin1983error}
D.~L. Ragozin.
\newblock Error bounds for derivative estimates based on spline smoothing of
  exact or noisy data.
\newblock {\em J. Approx. Theory}, 37(4):335--355, 1983.

\bibitem{schafer2011savitzky}
R.~W. Schafer.
\newblock What is a {S}avitzky-{G}olay filter? [lecture notes].
\newblock {\em IEEE Signal processing magazine}, 28(4):111--117, 2011.

\bibitem{strang1999discrete}
G.~Strang.
\newblock The discrete cosine transform.
\newblock {\em SIAM Rev.}, 41(1):135--147, 1999.

\bibitem{strossner2024approximation}
C.~Str{\"o}ssner, B.~Sun, and D.~Kressner.
\newblock Approximation in the extended functional tensor train format.
\newblock {\em Adv. in Comput. Math.}, 50(3):54, 2024.

\bibitem{Chebfun2}
A.~Townsend and L.~N. Trefethen.
\newblock An extension of {C}hebfun to two dimensions.
\newblock {\em SIAM J. Sci. Comp.}, 35(6):C495--C518, 2013.

\bibitem{trefethen2000spectral}
L.~N. Trefethen.
\newblock {\em Spectral Methods in MATLAB}.
\newblock SIAM, Philadelphia, 2000.

\bibitem{trefethenatap}
L.~N. Trefethen.
\newblock {\em Approximation Theory and Approximation Practice, Extended
  Edition}.
\newblock SIAM, Philadelphia, 2019.

\bibitem{tsybakov}
A.~B. Tsybakov.
\newblock {\em Introduction to Nonparametric Estimation}.
\newblock Springer, New York, 2009.

\bibitem{vershynin2018high}
R.~Vershynin.
\newblock {\em High-dimensional probability: An introduction with applications
  in data science}.
\newblock Cambridge University Press, Cambridge, 2018.

\bibitem{wainwright2019high}
M.~J. Wainwright.
\newblock {\em High-Dimensional Statistics: A Non-Asymptotic Viewpoint}.
\newblock Cambridge University Press, Cambridge, 2019.

\bibitem{wasserman2006all}
L.~Wasserman.
\newblock {\em All of Nonparametric Statistics}.
\newblock Springer, New York, 2006.

\bibitem{wasserman2013all}
L.~Wasserman.
\newblock {\em All of Statistics: a Concise Course in Statistical Inference}.
\newblock Springer, New York, 2013.

\bibitem{wilber2022data}
H.~Wilber, A.~Damle, and A.~Townsend.
\newblock Data-driven algorithms for signal processing with trigonometric
  rational functions.
\newblock {\em SIAM J. Sci. Comp.}, 44(3):C185--C209, 2022.

\bibitem{yanagihara2010unbiased}
H.~Yanagihara and K.~Satoh.
\newblock An unbiased {$C_p$} criterion for multivariate ridge regression.
\newblock {\em Journal of Multivariate Analysis}, 101(5):1226--1238, 2010.

\bibitem{yanagihara2015consistency}
H.~Yanagihara, H.~Wakaki, and Y.~Fujikoshi.
\newblock A consistency property of the aic for multivariate linear models when
  the dimension and the sample size are large.
\newblock {\em Electronic Journal of Statistics}, pages 869--897, 2015.

\bibitem{yohai1979asymptotic}
V.~J. Yohai and R.~A. Maronna.
\newblock Asymptotic behavior of m-estimators for the linear model.
\newblock {\em The Annals of Statistics}, pages 258--268, 1979.

\bibitem{zhu2023convergencenearoptimalsamplingmultivariate}
W.~Zhu and Y.~Nakatsukasa.
\newblock Convergence and near-optimal sampling for multivariate function
  approximations in irregular domains via vandermonde with arnoldi.
\newblock {\em arXiv preprint arXiv:2301.12241}, 2023.

\bibitem{zhukweighted}
S.~Zhuk.
\newblock Weighted pseudoinverse of compact operators and differentiation of
  signals with random noise.
\newblock In {\em 23rd International Symposium on Mathematical Theory of
  Networks and Systems Hong Kong University of Science and Technology, Hong
  Kong}, pages 85--92, 2018.

\end{thebibliography}

\end{document}